\documentclass[11pt]{article}
\usepackage[margin =1in]{geometry}

\usepackage{amssymb,amsthm,amsmath}
\usepackage{enumerate}
\usepackage{hyperref}
\usepackage{cite}
\usepackage[capitalize]{cleveref}
\usepackage{enumitem}
\usepackage{color}
\usepackage{cancel}

\usepackage{graphicx}
\usepackage{subcaption}
\usepackage[title]{appendix}

\newcommand{\norm}[1]{\left \| #1 \right \|}

\newcommand{\inclu}[0] {\ar@{^{(}->}}

\newcommand{\spann}{\text{span}}

\newcommand{\gph}{{\rm gph}\,}
\newcommand{\diag}{{\rm diag}}
\newcommand{\dist}{{\rm dist}}

\newcommand{\cS}{\mathcal{S}}

\newcommand{\EE}{\mathbb{E}}

\newcommand{\lip}{\mathrm{lip}}

\newcommand{\RR}{\mathbb{R}}

\newcommand{\rank}{\mathrm{rank}}
\newcommand{\supp}{{\mathrm{supp}}} 
\newcommand{\op}{\mathrm{op}}

\renewcommand{\SS}{{\mathbb{S}}}

\newcommand{\abs}[1]{\left| #1 \right|}

\newcommand{\argmin}{\operatornamewithlimits{argmin}}

\newtheorem{thm}{Theorem}[section]
\newtheorem{theorem}{Theorem}[section]

\newtheorem{proposition}[thm]{Proposition}

\newtheorem{lemma}[thm]{Lemma}

\newtheorem{claim}{Claim}

\newcommand{\cO}{\mathcal{O}}

\usepackage{mathtools}
\DeclarePairedDelimiter{\dotp}{\langle}{\rangle}
\usepackage[boxruled]{algorithm2e}

  \title{The nonsmooth landscape of blind deconvolution}

  \author{Mateo D\'iaz\thanks{\texttt{people.cam.cornell.edu/md825/}} \\ 
  Center for Applied Mathematics\\
  Cornell University} 
  
\begin{document}

\maketitle

\begin{abstract}
The blind deconvolution problem aims to recover a rank-one matrix from a set of rank-one linear measurements. Recently, Charisopulos et al. \cite{charisopoulos2019composite} introduced a nonconvex nonsmooth formulation that can be used, in combination with an initialization procedure, to provably solve this problem under standard statistical assumptions. In practice, however, initialization is unnecessary. As we demonstrate numerically, a randomly initialized subgradient method consistently solves the problem. In pursuit of a better understanding of this phenomenon, we study the random landscape of this formulation. We characterize in closed form the landscape of the population objective and describe the approximate location of the stationary points of the sample objective. In particular, we show that the set of spurious critical points lies close to a codimension two subspace. In doing this, we develop tools for studying the landscape of a broader family of singular value functions, these results may be of independent interest.
\end{abstract}

\section{Introduction}

An increasing amount of research has shown how matrix recovery problems, which in the worst case are hard, become tractable under appropriate statistical assumptions. Examples include phase retrieval \cite{wirt_flow,waldspurger2018phase,eM}, blind deconvolution \cite{li2016rapid,charisopoulos2019composite}, matrix sensing \cite{park2016non,MR3480732}, matrix completion \cite{MR3565131,mat_comp_min}, and robust PCA \cite{rob_cand,yi2016rpca}, among others \cite{boumal2016non,li2018nonconvex_robust,charisopoulos2019low,jusunblog}. Convex relaxations have proven to be a great tool to tackle these problems, but they often require lifting the problem to a higher dimensional space and consequently end up being computationally expensive. Thus, focus has shifted back to iterative methods for nonconvex formulations that operate in the natural parameter space. One of the difficulties of nonconvex optimization is that, in general, it is hard to find global minimizers. To overcome this issue, recent works have suggested two stage methods: One starts by running an \emph{initialization procedure} -- usually based on spectral techniques -- and then refines the solution by warm-starting a \emph{local search method} that minimizes a nonconvex formulation. This thread of ideas has proven very successful, and we refer the reader to \cite{chi2018nonconvex} for a survey.

Initialization procedures are nontrivial to develop and can sometimes more expensive than the refinement stage. Thus, it is important to understand when initialization methods are superfluous. There are iterative methods, for specific problems, that provably converge to minimizers \cite{mat_comp_min,ge2017unified,boumal2016non,cifuentes2019burer,lee2016gradient}. Analysis of these methods are of two types: those based on studying the iterate sequence \cite{kwon2018global,ding2018leave,zhong2018near}, and those based on characterizing the landscape of smooth loss functions \cite{ge2017unified, phase_nonconv,ling2018landscape}. 

In this work, we study the landscape of a nonsmooth nonconvex formulation \eqref{p:1} for the (real) \emph{blind deconvolution} problem. Unlike the aforementioned works, we consider a nonsmooth loss, which presents fundamentally different technical challenges. We show that, as the number of measurements grow, the set of spurious stationary points converges to a codimension two subspace. This suggests that there is an extensive region with friendly geometry. 

The blind deconvolution problem aims to recover a pair of real vectors $(\bar w, \bar x) \in \RR^{d_1} \times \RR^{d_2}$ from a set of $m$ observations given by 
\begin{equation}
\label{eq:measurements}
y_i = \dotp{a_i, \bar w} \dotp{\bar x, b_i}\qquad \text{for all } i \in{1, \dots, m}
\end{equation} 
where $a_i$ and $b_i$ are known vectors for all indices. This problem has important applications in a variety of different fields, we describe two below.
\begin{enumerate}
\item[]\textbf{Signal processing.} The complex analogue of this problem is intimately linked to the problem of recovering a pair of vectors $(u,v)$ from the convolution $(Au) \ast (Bv)$, where $A$ and $B$ are tall-skinny matrices. In fact, when passed to the frequency domain, this problem becomes equivalent to the one mentioned above. This problem has applications in image deblurring and channel protection with random codes \cite{ahmed2014blind,phase_nonconv}. 
\item[]\textbf{Shallow neural networks.} Solving this problem is equivalent to learning the weights of a shallow neural network with bilinear activation functions. Taking $\big\{((a_i,b_i), y_i)\big\}$ as training data, writing the output of the network as $y = \sigma( a^\top w, b^\top x)$, with $(w,x) \in \RR^{d_1 \times d_2}$, and the setting the activation function to $\sigma(z_1, z_2) = z_1 z_2.$
\end{enumerate}

To tackle the blind deconvolution problem, \cite{charisopoulos2019composite} proposed the following nonconvex nonsmooth formulation 
\begin{equation} \label{p:1}
\argmin_{w, x} f_S(w,x) \triangleq \frac{1}{m} \sum_{i=1}^m |\dotp{a_i, w} \dotp{x, b_i} - y_i|. 
\end{equation}
The authors of \cite{charisopoulos2019composite} designed a two-stage
method based on this formulation and showed that if the measuring
vectors, $a_i$ and $b_i,$ are i.i.d standard Gaussian, then their
algorithm converges rapidly to a solution whenever
$m \gtrsim (d_1+d_2).$ \footnote{This is information-theoretically
  optimal up to constants \cite{choudhary2014sparse,kech2017optimal}.}
Nonetheless, experimentally the initialization stage seems to be
superfluous. Indeed, a simple randomly-initialized subgradient
algorithm is successful at solving the problem most of the time
provided that $m/(d_1+d_2)$ is big enough, see the experiments in the
last section for support of this claim.

\subsection{Main contributions}
Aiming to get a better understanding of the
high-dimensional geometry of $f_S$, we study the landscape of $f_S$
when $A$ and $B$ are standard Gaussian random matrices. Following the
line of ideas in \cite{davis2017nonsmooth}, we think of $f_S$ as the
empirical average approximation of the \emph{population objective}
\begin{equation*}
f_P(w,x) \triangleq \EE f_S(w,x) = \EE \left( | a^\top(wx - \bar w \bar x) b^\top|\right),
\end{equation*}
where $a \in \RR^{d_1}$ and $b \in \RR^{d_2}$ are standard Gaussian vectors. From now on, we will refer to $f_S$ as the \emph{sample objective.}
The rationale is simple: we will describe the stationary points of $f_P$, then we will
prove that the graph of the subdifferential $\partial f_S$ concentrates around the graph of $\partial f_P$ and 
combine these to describe the landscape of $f_S.$\footnote{We will give a formal definition of $\partial f$ in Section~\ref{sec:notation}.} This strategy allows us to show that the set of spurious stationary points converges to a 
codimension two subspace at a controlled rate. We remark that these results are geometrical and not computational. 

Before we go on, let us observe that one can only wish to recover the
pair $(\bar w, \bar x)$ up to scaling. In fact, the measurements
\eqref{eq:measurements} are invariant under the mapping
$(w, x) \mapsto (\alpha w, x/\alpha)$ for any $\alpha \neq 0.$ Hence
the set of solutions of the problems is defined as
$$\cS \triangleq \left\{\left(\alpha \bar w, \bar x /\alpha\right) \mid
  \alpha \in \RR \setminus \{0\}\right\}.$$

\paragraph{Population objective.} Interestingly, the population objective only depends on $(w,x)$ through the 
singular values of the rank two matrix $X := wx^\top - \bar w \bar x^\top.$ We show this function can be written as 
\begin{equation*}
f_P(w,x) =  \sigma_{\max}(X) \sum_{n=0}^\infty \left(\frac{(2n)!}{2^{2n}(n!)^2}\right)^2\frac{
\left(1- \kappa^{-2}(X) \right)^n}{1-2n}
\end{equation*}
where $\kappa(X) = {\sigma_{\max}(X)}/{\sigma_{\min}(X)}$ is the condition number of $X.$ We characterize the stationary points of a broad family of spectral functions, containing $f_P$. By specializing this characterization we find that the stationary points of the population objectives are exactly
\[\cS\cup \{(w,x) \mid  \dotp{w, \bar w} = 0, \dotp{x, \bar x} = 0, \text{and } wx^\top =0\},\]
revealing that the set of extraneous critical points of $f_P$ is the subspace $(\bar w, 0)^\perp \cap (0, \bar x)^\perp.$

\paragraph{Sample objective.} Equipped with a quantitative version of Attouch-Wets' convergence theorem proved in \cite{davis2017nonsmooth}, we show that with high probability any stationary point of $f_S$ in a bounded set satisfies at least one of the following
\[\|(w,x)\| \leq \Delta \|(\bar w,\bar x)\|,\;\; \|wx^\top - \bar w \bar x^\top\| \leq \Delta \|\bar w \bar x^\top\|, \;\; \text{or} \;\;\left\{
    \begin{array}{ll}
  |\dotp{w,\bar w}| &\leq \Delta \|(w,x)\| \|\bar w\|, \\
  |\dotp{x,\bar x}| &\leq \Delta \|(w,x)\| \|\bar x\|. 
  \end{array}
  \right.\]
provided that $m \gtrsim d_1 + d_2,$ where $\Delta = \widetilde \cO\left(\frac{d_1+d_2}{m} \right)^{\frac{1}{8}}$.\footnote{ Where $\widetilde \cO$ hides logarithmic terms.} Intuitively this means, that as the ratio $(d_1+d_2)/m$ goes to zero, the stationary points lie closer and closer to three sets: the singleton zero, the set of solutions $\cS$, and the subspace $(\bar w, 0)^\perp \cap (0, \bar x)^\perp.$

\subsection{Related work}
There is a vast recent literature on blind deconvolution. A variety of algorithmic solutions have been proposed, including convex relaxations \cite{ahmed2014blind,ahmed2018blind}, Riemannian optimization methods \cite{HuangHand17}, gradient descent algorithms \cite{li2016rapid,ma2017implicit},  and nonsmooth procedures \cite{charisopoulos2019composite}. Related to this work, the authors of \cite{zhang2017global,kuo2019geometry} studied variations of the blind deconvolution problem via landscape analysis; their approach is based on smooth formulations and therefore their tools are of a different nature. Besides algorithms, researchers have also been interested in information-theoretical limits of the problem under different assumptions\cite{choudhary2014sparse,li2016identifiability,kech2017optimal}.  

On the other hand, the study of the high-dimensional landscape of
nonconvex formulations is an emergent area of research. Examples for
smooth formulations include the analysis for phase retrieval
\cite{phase_nonconv}, matrix completion \cite{mat_comp_min}, robust
PCA \cite{ge2017unified}, and synchronization networks
\cite{ling2018landscape}. The majority of these results focus on using
second order information to show that under reasonable assumptions the
formulations do not exhibit spurious stationary points. The machinery
developed for nonsmooth formulations is based on different ideas and
is more case-oriented. Despite there are remarkable examples
\cite{davis2017nonsmooth,josz2018theory,fattahi2018exact,bai2018subgradient}. Closer to our
work is the paper \cite{davis2017nonsmooth}; the authors of
this article studied a similar nonsmooth formulation for the phase
retrieval problem, which can be regarded as a symmetric analogue of
blind deconvolution.
\subsection{Outline}
The agenda of this paper is as follows: Section~\ref{sec:notation} introduces notation and some basic results we require. Sections~\ref{sec:population} and~\ref{sec:sample} present the results on the landscape of the population and sample objectives, respectively. In Section~\ref{sec:experiments}, we present computational experiments corroborating the conclusions of our theory. We close with a brief discussion and future research directions in Section~\ref{sec:dis}. Many of the arguments are technical and consequently we defer most of the proofs to the appendices.

\section{Preliminaries} \label{sec:notation}

We will follow standard notation. The symbols $\RR$ and $\RR_+$ denote the real line and the nonnegative reals, respectively. The set of extended reals $\RR \cup \{+\infty\}$ is written as $\overline \RR.$ We always endow $\RR^d$ with its standard inner product, $\dotp{x,y} = x^\top y$, and its induced norm $\|x\| = \sqrt{\dotp{x,x}}$. We also use $\|x\|_1 = \sum |x_i|$ to denote the $\ell_1$-norm. 
For a set $S \subseteq \RR^d$, we denote the distance from a point $x$ to the set by $\dist(x,Q) = \inf_{y \in Q} \|x - y\|$. For any pair of real-valued functions $f, g: \RR^n \rightarrow \RR$, we say that $f \lesssim g $ if there exists a constant $C$ such that $f \leq C g.$ Moreover, we write $f \asymp g$ if both $f \lesssim g$ and $g \lesssim f.$ 

The adjoint of a linear operator $A : \RR^d \rightarrow \RR^n$ is indicated by 
$A^\top: \RR^n \rightarrow \RR^d.$ Assuming $d\leq n$, the 
map $\sigma: \RR^{n \times d} \rightarrow \RR^{d}_+$ returns the vector of 
ordered singular values of a matrix with $\sigma_1(A) \geq \sigma_2(A) \geq  
\dots \geq \sigma_{d}(A)$. We will use the symbols $\|A\|_\op = \sigma_1(A)$ 
and $\|A\|_F = \|\sigma(A)\|_2$ to indicate the operator and Frobenius norm, respectively. When 
not specified it is understood that $\|A\| := \|A\|_\op.$ We will use the symbol $O(d)$ to denote the set of $d \times d$ orthogonal matrix.
\paragraph{Variational analysis.} Since we will handle nonsmooth functions, we need a definition of generalized derivatives. We refer the interested reader to some excellent references on the subject \cite{RW98,mord1,Borwein-Zhu}. Let $f: \RR^d \rightarrow \overline \RR$ be a lower semicontinuous proper function and $\bar x$ be a point. The \emph{Fr\'echet subdifferential} $\widehat \partial f(x)$ is the set of all vectors $\xi$ for which 
\[f(x) \geq f(\bar x) + \dotp{\xi, x- \bar x} + o(\|x-\bar x\|)\qquad \text{as }x\rightarrow \bar x. \]
Intuitively, $\xi \in \partial f(\bar x)$ if the function $x\mapsto f(\bar x) + \dotp{\xi, x- \bar x}$ locally minorizes $f$ up to first order information. Unfortunately, the set-valued mapping $x \mapsto \widehat \partial f(x)$ lacks some desirable topological properties. For this reason it is useful to consider an extension. The \emph{limiting subdifferential} $\partial f (\bar x)$ is the set of all $\xi$ such that there are sequences $(x_n)$ and $(\xi_n)$ with $\xi_n \in \widehat \partial f(x_n)$ satisfying $(x_n, f(x_n), \xi_n) \rightarrow (\bar x, f( \bar x), \xi)$. It is well-known that $\partial f( \bar x)$ reduces to the classical derivative when $f$ is Fr\'echet differentiable and that for $f$ convex, $\partial f(\bar x)$ is equal to the usual convex subdifferential
\[\xi \in \partial f( \bar x) \qquad \iff \qquad  f(x) \geq f(\bar x) + \dotp 
{\xi , x - \bar x} \;\;\; \forall x.\]We say that a point $\bar x$ is 
\emph{stationary} if $0 \in \partial f(\bar x).$ The \emph{graph} of $\partial 
f$ is given by $\gph \partial f = \{(x,\zeta) \mid \zeta \in \partial f(x)\}.$ 

For $\rho > 0,$ we say that $f$ is $\rho$-weakly convex if the regularized function $f + \frac{\rho}{2}\|\cdot\|^2_2$ is convex. This encompasses a broad class of functions: Any function that can be decomposed as $f = h \circ g$, where $h : \RR^m \rightarrow \RR$ is a Lipschitz convex function and $g : \RR^d \rightarrow \RR^m$ is smooth map, is weakly convex. It is worth noting that for functions that can be decomposed in this fashion, the chain rule \cite[Theorem 10.6]{RW98} yields $\partial f(x) = \nabla g(x)^\top \partial h ( g(x))$ for all $x.$   

\paragraph{Singular value functions.} For a pair of dimensions $d_1, d_2$ we will denote $d = \min\{d_1, d_2\}.$ A function $f: \RR^d \rightarrow \RR$ is \emph{symmetric} if $f(\pi x) = f(x)$ for any permutation matrix $\pi\in \{0,1\}^{d \times d}$. Additionally, a function $f$ is \emph{sign invariant} if $f(s x) = f(x)$ for any diagonal matrix $s \in \{-1,0,1\}^{d \times d}$ with diagonal entries in $\{\pm 1\}.$ We say that $f_\sigma : \RR^{d_1 \times d_2} \rightarrow \RR$ is a \emph{singular value function} if it can be decomposed as $f_\sigma = (f \circ \sigma)$ for a symmetric sign invariant function $f.$ A simple and illuminating example is the Frobenius norm, since $\|A\|_F = \| \sigma(A)\|_2.$ This type of function has been deeply studied in variational analysis \cite{der,eval,lewis2005nonsmooth}.

A pair of matrices $X$ and $Y$ in $\RR^{d_1 \times d_2}$ have a \emph{simultaneous ordered singular value decomposition} if there exist matrices $U \in O(d_1)$ and $V \in O(d_2)$ such that $X = U \diag(\sigma(X)) V^\top$ and $Y = U \diag(\sigma(Y)) V^\top.$ We will make great use of the following remarkable theorem. 

\begin{theorem}[Theorem 7.1 in \cite{lewis2005nonsmooth}]\label{thm:spectral_sub}
The limiting subdifferential of a singular value function $f_\sigma = f \circ \sigma$ at a matrix $M \in \RR^{d_1 \times d_2}$ is given by 
\begin{equation}
\partial f_\sigma (M) = \{U \diag (\zeta) V^\top \mid \zeta \in \partial f (\sigma(M)) \text{ and } U\diag \left(\sigma(M)\right) V^\top = M \}.
\end{equation}
Hence $M$ and any of its subgradients have simultaneous ordered singular value decomposition.
\end{theorem}

\begin{figure}[!t]
  \centering
    \begin{subfigure}[b]{0.45\textwidth}
      \centering\includegraphics[width=\textwidth]{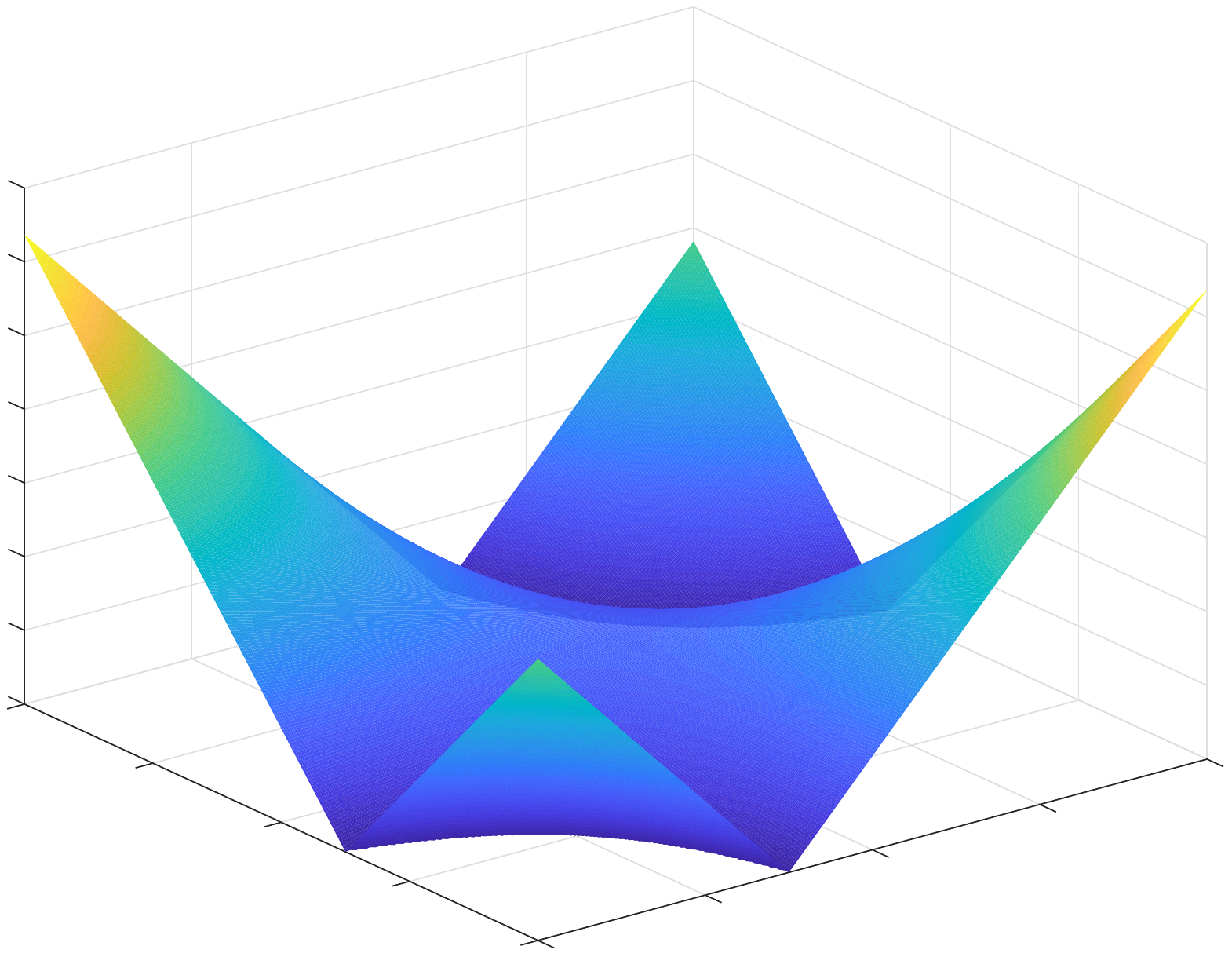}
            \label{fig:contour1}
    \end{subfigure}~
    \begin{subfigure}[b]{0.45\textwidth}
      \centering\includegraphics[width=\textwidth]{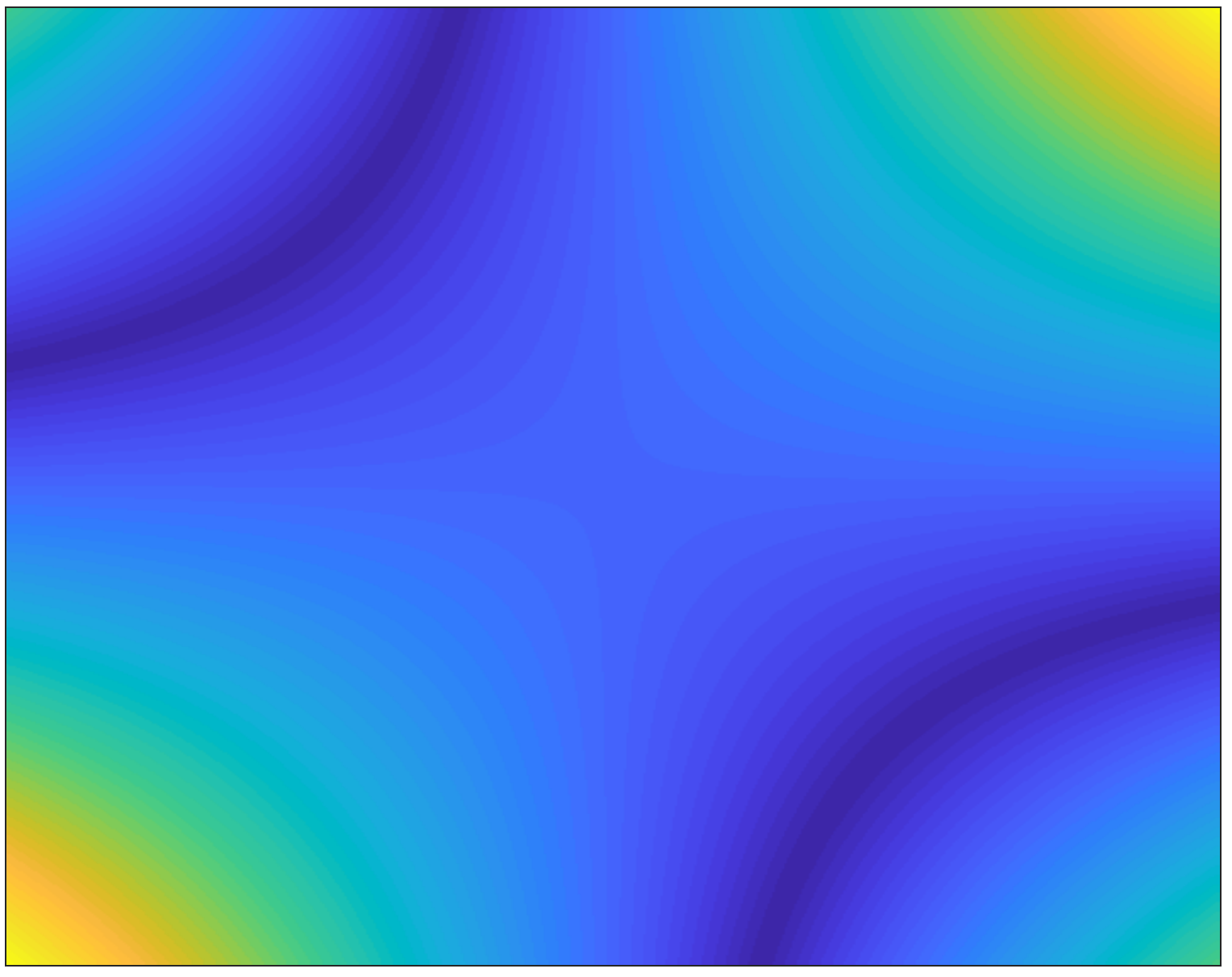}
      \label{fig:contour2}
    \end{subfigure}
\caption{Population objective $d_1 = d_2 = 1.$} \label{fig:4}
\end{figure}

\section{Population objective} \label{sec:population}

In this section we study the population objective $f_P$. A first important observation is that this function is a singular value function. Indeed, if we set $X = wx^\top- \bar w \bar x^\top$ then due to the orthogonal invariance of the Gaussian distribution we get 
\begin{equation} \label{eq:fP}
  f_P(w,x) = \EE |a^\top U\diag(\sigma(X))V^\top b| = \EE|\sigma_1(X) a_1 b_1 + \sigma_2(X) a_2 b_2|,
\end{equation} 
where of course $U\diag(\sigma(X))V^\top$ is the singular value decomposition of $X.$ This simple observation leads to our first result,  a closed form characterization of this function in terms $\sigma(X).$ We defer the proof to Appendix~\ref{subsec:proof_prop_closed_form}.

\begin{proposition}[Population objective]\label{prop:closed_form_population}
The population objective can be written as 
\begin{equation}
f_P(w,x) = \sigma_{\max}(X) \sum_{n=0}^\infty \left(\frac{(2n)!}{2^{2n}(n!)^2}\right)^2\frac{
\left(1- \kappa^{-2}(X) \right)^n}{1-2n}
\end{equation}
where $\kappa(X) = {\sigma_{\max}(X)}/{\sigma_{\min}(X)}$ is the condition number of $X.$
\end{proposition}

When the signal $(\bar w, \bar x)$ lives in  $\RR^2$ the landscape of the population objective is rather simple, the only critical points are the solutions and zero, see Figure~\ref{fig:4}. This is not the case in higher dimensions where an entire subspace of critical points appear. In the reminder of this section, we develop tools to describe the critical points of a broad class of functions and we then specialize these results to the blind deconvolution population objective \eqref{eq:fP}.

\subsection{Landscape analysis for a class of singular value functions}
From now on we consider an arbitrary function $g: \RR^{d_1} \times \RR^{d_2} \rightarrow \RR$ for which there exists a rank one matrix $\bar w \bar x^\top$ and a singular value function $f_\sigma$ satisfying 
\[g(w,x) = f_\sigma (wx^\top - \bar w \bar x^\top) = f \circ \sigma \left(wx^\top - \bar w \bar x^\top\right).\]
This gives us two useful characterizations of $g$ that we will use throughout. In the following section we will see a way of recasting $f_P$ in this form. 

A simple application of the chain rule yields 
\begin{equation}
\partial g(w,x) = \left\{ \begin{bmatrix}
  Y x \\ 
  Y^\top w
\end{bmatrix} \, \Big| \, Y \in \partial f_{\sigma} (X)\right\}.
\end{equation}
Notice that we already have a description of $\partial f_\sigma (X)$ given by Theorem~\ref{thm:spectral_sub}, that is $Y \in \partial f_\sigma (X)$ if and only if there exists matrices $U \in O(d_1)$ and $V \in O(d_2)$ satisfying 
\begin{equation} \label{eq:descY} \sigma(Y) \in \partial f( \sigma(X)), \qquad Y = U \diag(\sigma(Y)) V^\top, \qquad \text{and} \qquad X = U \diag( \sigma(X)) V^\top.\end{equation} 
Equipped with these tools we derive the following result regarding the critical points of $g.$ We defer a proof to Appendix~\ref{proof:generalPopulation}.

\begin{theorem} \label{cor:general-landscape}
Suppose that $(w,x)$ is a stationary point for $g,$ i.e. $Yx = 0, Y^\top w = 0.$ Then at least one of the following conditions hold: 
\begin{enumerate}
  \item \textbf{Small objective.} $g(w,x) \leq g(\bar w, \bar x),$
  \item \textbf{Zero.} $(w,x) = 0,$
  \item \textbf{One zero component.} $\dotp{w,\bar w} = \dotp{x, \bar x} = 0$, $wx^\top = 0,$ and (assuming that $x$ is not zero) $Yx = 0$ (similarly for $w$). 
  \item \textbf{Small product norm.} $\dotp{w,\bar w} = \dotp{x, \bar x} = 0$, $\rank(Y) = 1$, and $0 < \|wx^\top \| < \|\bar w \bar x^\top\|.$
\end{enumerate}
Moreover, if $(\bar w,\bar x)$ minimizes $g$, then $(w,x)$ is a critical point if, and only if, it satisfies 1, 2, 3, or 4 for some $Y \in \partial f_\sigma(X).$  
\end{theorem}
\subsection{Landscape of the population objective}
Our goal now is to apply Theorem~\ref{cor:general-landscape} to describe the landscape of $f_P.$ In order to do it we need to write $f_P(w,x)= f \circ \sigma (X) $ with $f: \RR^d \rightarrow \RR$ a symmetric sign-invariant convex function. An easy way to do this is to define 
\[f(s_1, \dots, s_d) = \EE \left(\left|\sum_{i=1}^d a_i b_i s_i \right|\right).\]  

To use Theorem~\ref{cor:general-landscape}, we need to study $\partial f.$ The next lemma shows that the function is actually differentiable at every point but zero. We defer the proof of this result to Appendix~\ref{subsec:proof_lemma_partial_derivatives}.

\begin{lemma}\label{lemma:partial_derivatives}
For any nonzero vector $s \in \RR^d_+ \setminus \{0\},$ the partial derivatives of $f$ satisfy 
 \begin{equation} \label{eq:gradient}
  \frac{\partial f}{\partial s_j}(s) = \sqrt{\frac{2}{\pi}}s_j \; \EE \left[{a_j^2}{\left(\sum_i^d (a_is_i)^2\right)^{-\frac{1}{2}}} \right].
 \end{equation}
\end{lemma}

This lemma gives us the final tool to derive the main theorem regarding the landscape of $f_P.$
\begin{theorem}\label{theo:mainPopulation}
The set of critical points of the population objective $g_P$ is exactly 
\[\cS \cup \{0\} \cup \{(w,x) \mid \dotp{w, \bar w} = 0, \dotp{x, \bar x} = 0, \text{and } wx^\top =0\}.\]
\end{theorem}
\begin{proof}
Notice that $(\bar w, \bar x)$ minimizes the population objective $f_P$, therefore Theorem~\ref{cor:general-landscape} gives a complete description of the critical points. Let us examine each one of the conditions in this theorem. 

The points in $\{(w,x) \mid wx^\top = \bar w \bar x^\top \}$ and $\{0\}$ are contained in the set of stationary points because they satisfy the first and second condition, respectively. 

Now, let $(w,x) \in \{\bar w\}^\perp \times \{\bar x\}^\perp$ such that $wx^\top = 0.$ Thus, the matrix $X$ is rank $1$, and consequently~\eqref{eq:gradient} reveals that that any $Y \in \partial f_\sigma(X)$ satisfies $\sigma(Y) = \nabla f(\sigma(X)) = (2/\pi, 0 , \dots, 0)$. Therefore, due to \eqref{eq:descY}, we get $Y = \frac{2}{\pi} \bar w \bar x^\top / \|\bar w\|\|\bar x\|.$ Without loss of generality, assume $x$ is not zero. Then, $\norm{Yx} = \frac{2}{\pi \|\bar x\|}|\dotp{x,\bar x} | = 0$ and, consequently, $(w,x)$ is stationary. 

On the other hand, let $(w,x) \in \{\bar w\}^\perp \times \{\bar x\}^\perp$ such that $0 <\|wx^\top \| < \|\bar w\bar x^\top \|$. Therefore, the matrix $X$ is rank $2$ and so~\eqref{eq:gradient} gives that $\sigma_2(Y) > 0$ for all $Y \in \partial f_\sigma(X).$ Hence, $(w,x)$ is not a stationary point, giving the result.
\end{proof}
\section{Sample objective} \label{sec:sample}

In this section we describe the approximate locations of the critical points of the sample objective. Unlike in the  smooth case, nonsmooth losses do not exhibit point-wise concentration of the subgradients, or in other words, $\partial f_S(w,x)$ doesn't converges to $\partial f_P(w,x)$. To overcome this obstacle, we show that the graph of $\partial f_S$ approaches that of $\partial f_P$ at a quantifiable rate. This intuitively means that if $(w,x)$ is a critical point of $f_S,$ then nearby there exists a point $(\widehat w, \widehat x)$ with $\dist(0, \partial f_P(\widehat w, \widehat x))$ small.

The following result can be regarded as an analogous version of Theorem~\ref{theo:mainPopulation} for the sample objective. The reasoning behind this theorem is similar: we first develop theory for a broad class of functions, and then specialize it to $f_S.$ However, the proof of this result is more involved and will require us to study the location of epsilon critical points of the population. We defer the development of these arguments and the proof of the next result to Appendices~\ref{app:approximate} and~\ref{app:Sample}, respectively.

\begin{theorem}\label{theo:sample_objective}
Consider the sample objective \eqref{p:1} generated with two Gaussian matrices $A$ and $B$. For any fixed $\nu > 1$ there exist numerical constants $c_1, c_2, c_3> 0$ such that if $m \geq c_1(d_1+d_2+1)$, then with probability at least $1-c_2\exp(- c_3 (d_1+ d_2+1))$, every stationary point $(w,x)$ of $f_S$ for which $\|(w,x)\| \leq \nu \|(\bar w, \bar x)\|$ satisfies at least one of the following conditions 
\begin{enumerate}
  \item\textbf{Near zero.} \[\|(w,x)\| \leq   \|(\bar w, \bar x)\| \Delta,\]
  \item\textbf{Near a solution.} \[ \| wx^\top - \bar w \bar x^\top\| \lesssim (\nu^2+1)\|\bar w \bar x^\top\| \Delta,\]
\item\textbf{Near orthogonal.}
\[\left\{
    \begin{array}{rl}
  |\dotp{w,\bar w}| &\lesssim (\nu^2 + 1) \|(w,x)\| \|\bar w\| \Delta, \\
  |\dotp{x,\bar x}| &\lesssim (\nu^2 + 1)\|(w,x)\| \|\bar x\| \Delta. 
  \end{array}
  \right.
\]
\end{enumerate}
where $\Delta = \left( \frac{d_1+d_2 +1}{m} \log\left(\frac{m}{d_1+d_2+1}\right) \right)^{\frac{1}{8}}.$ 
\end{theorem}

We remark that one can further prove that with high probability, there exists a neighborhood around the solutions set $\cS$ in which the only critical points are the solutions \cite{charisopoulos2019composite}. Hence at the cost of potentially increasing $c_1$, the second condition can be strengthened to $(\bar w,\bar x) \in \cS.$

\section{Experiments} \label{sec:experiments}

In this section we empirically investigate the behavior of a randomly-initialized subgradient algorithm applied to \eqref{p:1}.\footnote{The experiments were run in a 2013 MacBook Pro with 2.4 GHz Intel Core i5 Processor and 8 GB of RAM.} It is known that well-tuned subgradient algorithms converge to critical points for any locally Lipschitz function \cite{Davis2019}. Further, these types of iterative procedures are computationally cheap, easy to implement, and widely used in practice. This makes them a great proxy for our purposes. For the experiments, we use Polyak's subgradient method, a classical algorithm known to exhibit rapid convergence near solutions for \emph{sharp weakly-convex functions} \cite{davis2018subgradient}.\footnote{It was shown in \cite{charisopoulos2019composite} that $f_S$ satisfies these assumptions with high probability provided $m \gtrsim (d_1+d_2)$.}  Polyak's method is an iterative algorithm given by 
\begin{equation}\label{eq:Polyak}
x_{k+1} \leftarrow x_k - \left( \frac{f(x_k) - \inf f}{\|g_k\|^2}  \right) g_k \qquad \text{with }g_k \in \partial f(x_k).
\end{equation} 
Notice that the step size requires us to know the minimum value, which in this case is exactly zero. Polyak's algorithm was used in \cite{charisopoulos2019low} as one of the procedures in the two-stage method for blind deconvolution.

In all the experiments the goal is to recover a pair of canonical vectors $(e_1, e_1) \in \RR^{d_1} \times \RR^{d_2}.$ Observe that this instance is a good representative of the average performance of the method due to the rotational invariance of the measurements. We evaluate the frequency of successful recovery of \eqref{eq:Polyak} using two different random initialization strategies: 
\begin{enumerate}
  \item (\textbf{Uniform over a cube}) We set $(w_0, x_0)$ to be an uniform vector on the cube $[-\nu, \nu]^{d_1 + d_2}.$ 
  \item (\textbf{Random Gaussian}) We set $w_0$ and $x_0$ to be distributed $\text{N}(0, \frac{\nu^2}{d_1} I_{d_1})$ and $\text{N}(0, \frac{\nu^2}{d_2} I_{d_2})$, respectively. This ensures that with high probability, both $\|w_0\|$ and $\|x_0\|$ are close to $\nu.$
\end{enumerate}
We generate phase transition plots for both initialization strategies by varying the value of $\nu$ and $C \triangleq m/(d_1+d_2)$ between $\{2^4, 2^5, \dots, 2^{10}\}$ and $\{1,2, \dots, 8\},$ respectively. For each choice of parameters we generate ten random instances $(w_0, x_0, A,B)$ and record in how many instances Polyak's method achieves a relative error smaller than $10^{-6}.$ The method stops whenever it reaches $100\,000$
iterations or the function value is less than $10^{-10}.$ We repeat these experiments for two different pairs of dimensions, $(d_1, d_2) \in \{(100, 50), (200, 100)\}$. The results are displayed in Figures~\ref{fig:1} and \ref{fig:2}. 

A first immediate observation is that the random initialization, the dimension, and the scaling parameter $\nu$ do not seem to be affecting the recovery frequency of the algorithm. The only parameter that controls the recovery frequency is $C.$ This is intuitively consistent with Theorem~\ref{theo:sample_objective}, since this parameter determines the concentration of spurious critical points around a subspace. Nonetheless, the effect of this parameter seems to be stronger in practice. Indeed, the probability of recovery exhibits a sharp phase transition at $C \sim 3$.

\paragraph*{Reproducible research.} All the results and code
implemented for these experiments are publicly available in \url{https://github.com/mateodd25/BlindDeconvolutionLandscape}.

\begin{figure}[!t]
  \centering
    \begin{subfigure}[b]{0.35\textwidth}
      \centering\includegraphics[width=\textwidth]{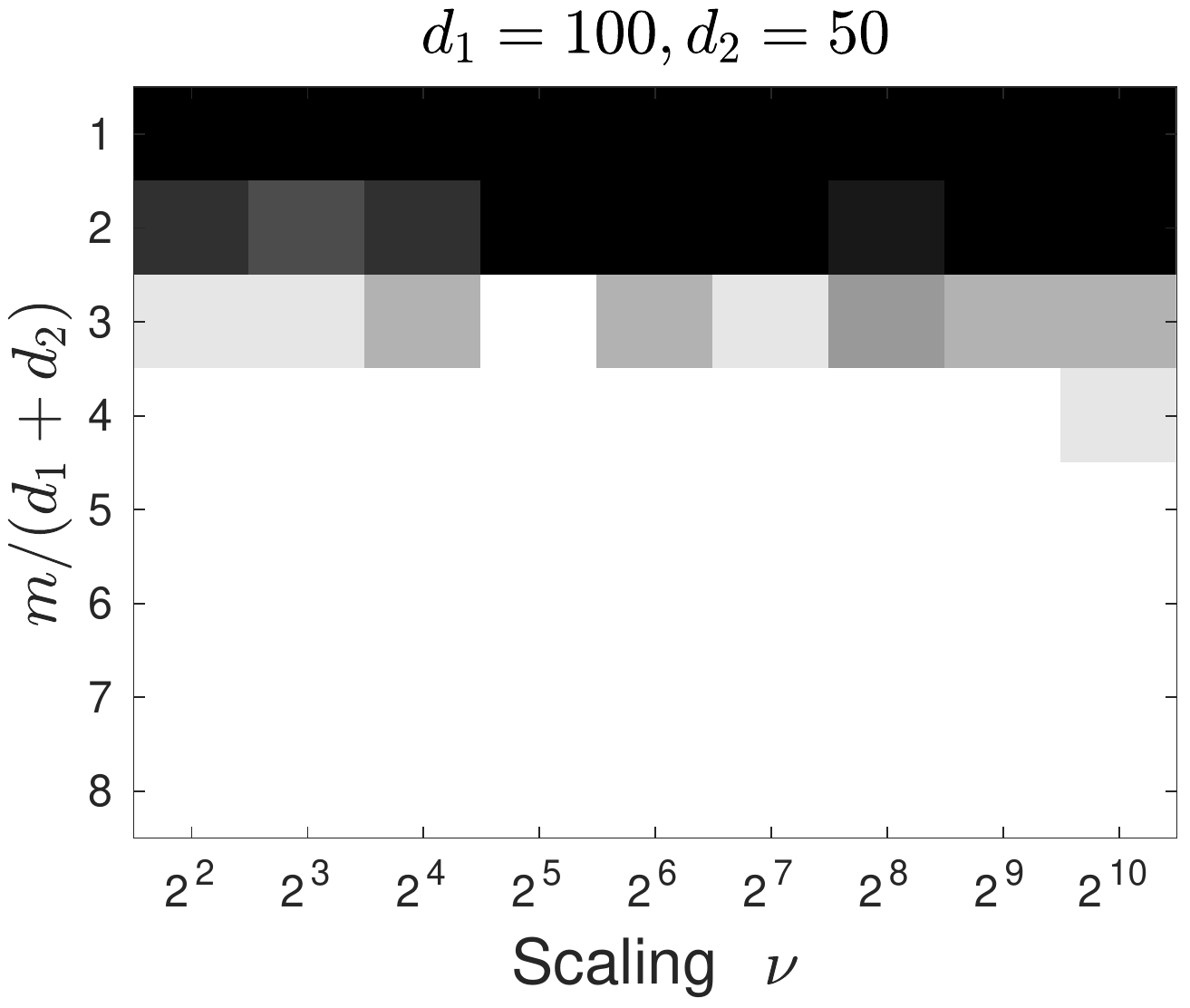}
            \label{fig:contour1}
    \end{subfigure}~
    \begin{subfigure}[b]{0.35\textwidth}
      \centering\includegraphics[width=\textwidth]{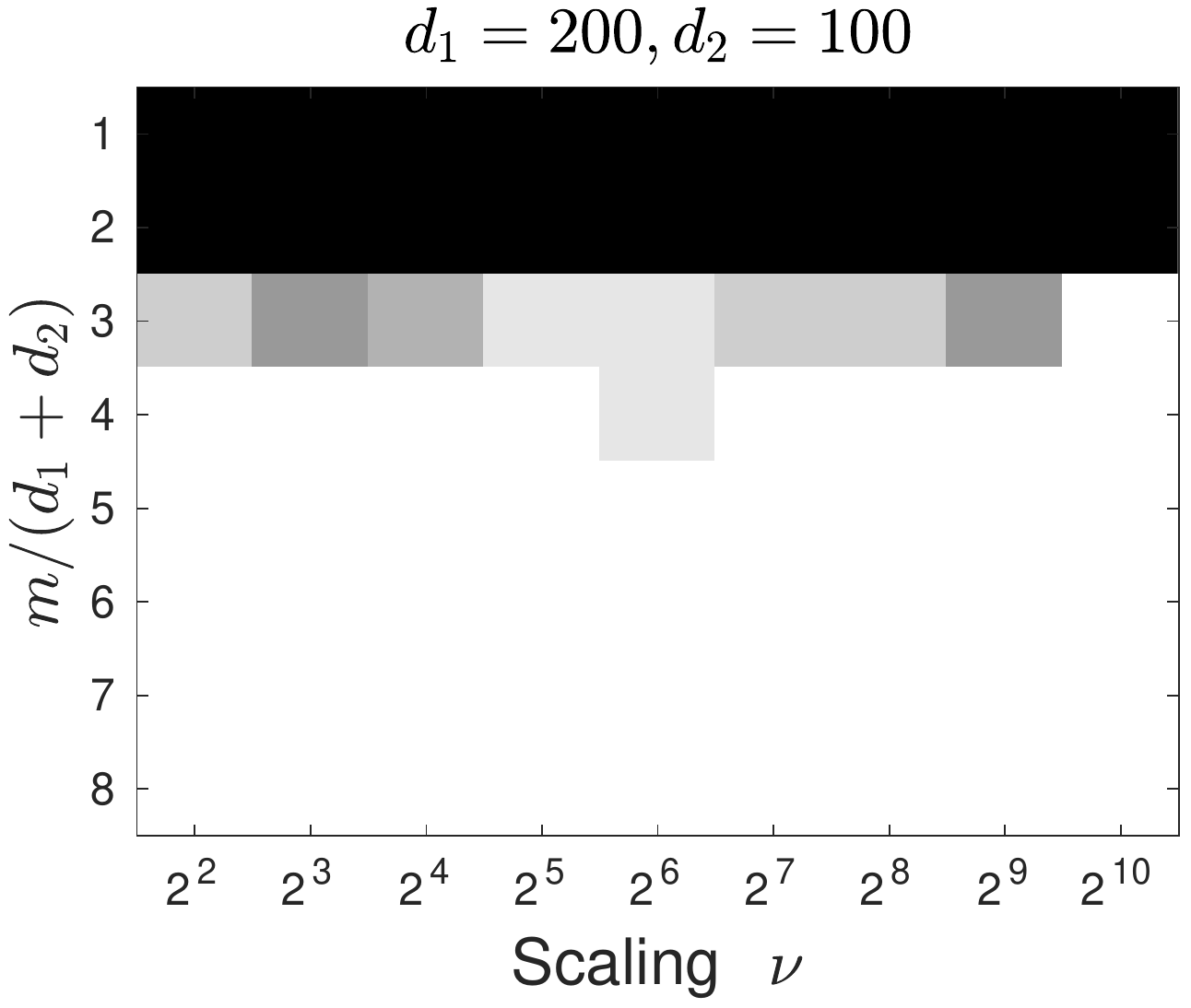}
      \label{fig:contour2}
    \end{subfigure}
    \begin{subfigure}[b]{0.35\textwidth}
      \centering\includegraphics[width=\textwidth]{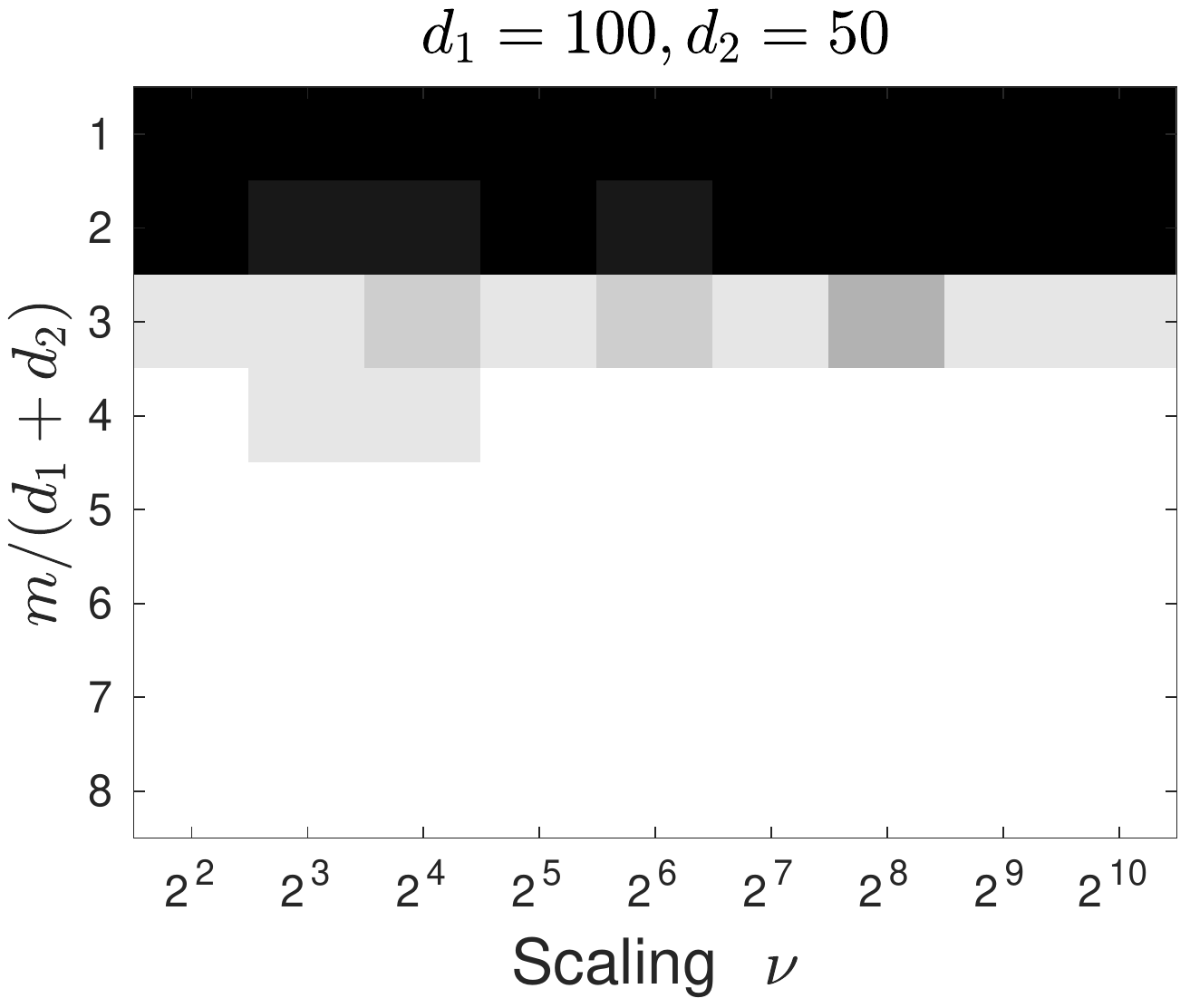}
            \label{fig:contour1}
    \end{subfigure}~
    \begin{subfigure}[b]{0.35\textwidth}
      \centering\includegraphics[width=\textwidth]{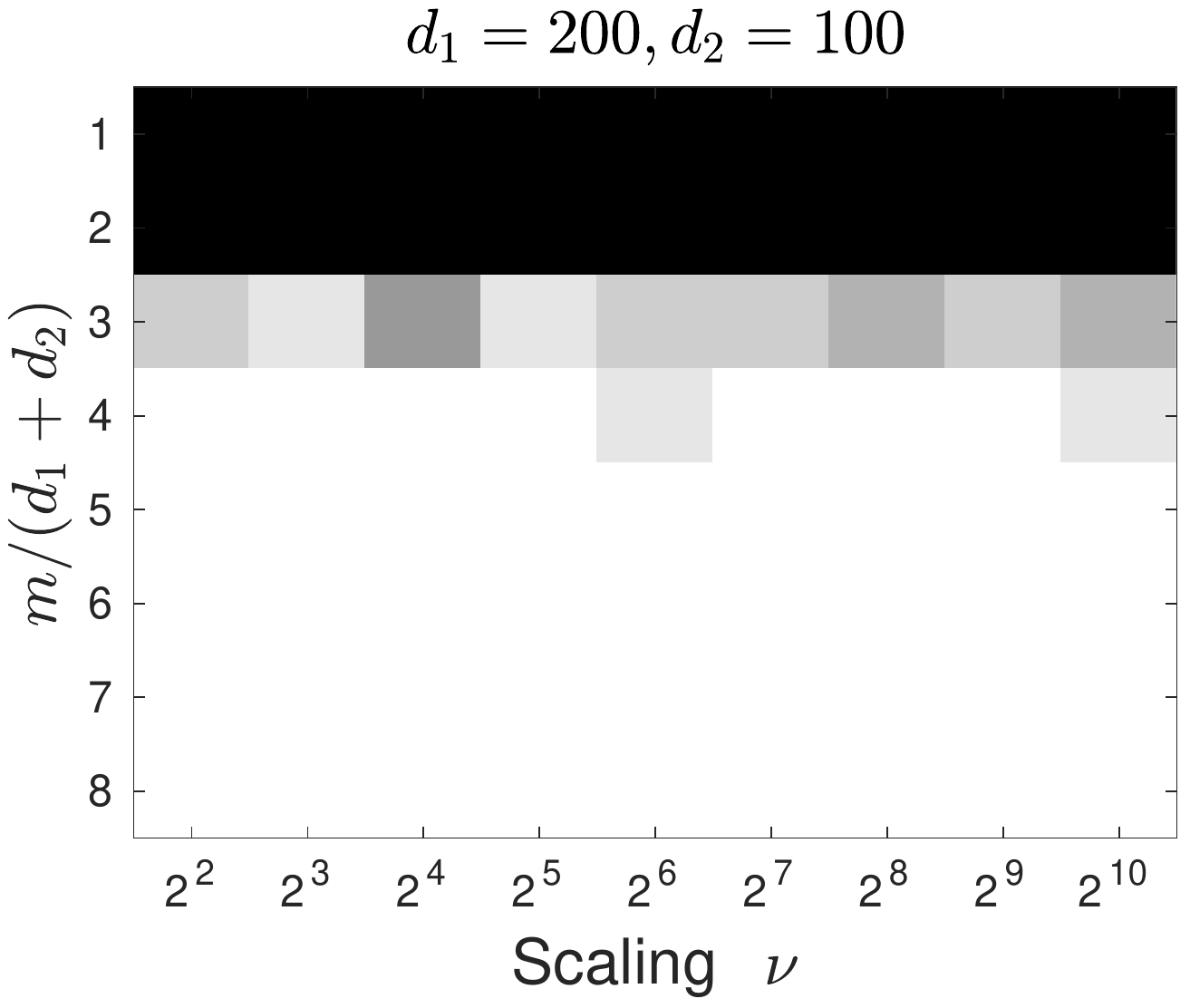}
      \label{fig:contour2}
    \end{subfigure}
\caption{Empirical probability of recovery with random initialization, first row with Gaussian distribution and second row with uniform over a cube. White denotes probability one and black denotes probability zero. Left and right images correspond to $(d_1,d_2) = (100, 50)$ and $(d_1,d_2) = (200, 100),$ respectively.} \label{fig:2}
\end{figure}

\section{Conclusions}
\label{sec:dis}
We investigated both the population and sample objectives of a
formulation for the blind deconvolution problem. We showed that in
both cases the set of spurious critical points are, or concentrate
near, a subspace of codimension two. Such concentration can be
measured in terms of the ratio of the dimension of the signal we wish
to recover over the number of measurements. This sheds light on the
fact that a randomly-initialized subgradient method converges to a
solution whenever this ratio is small enough. Our results, however, do
not entirely explain this behavior. It could be the case that we are
witnessing an instance of a more general phenomenon. It is known that
when the aforementioned ratio is small, the sample objective becomes
sharp weakly convex with high probability. It would be interesting to
know if for this type of function a well-tuned subgradient method
avoids spurious critical points. We leave this as an open question for
future research.

\section*{Acknowledgments}

I thank Jose Bastidas, Damek Davis, Dmitriy Drusvyatskiy, Robert Kleinberg, and Mauricio Velasco for insightful and encouraging conversations. Finally, I would like to thank my advisor Damek Davis for research funding during the completion of this work.

\bibliographystyle{unsrt}
\bibliography{bibliography2}
\appendix

\section{Proof of Proposition~\ref{prop:closed_form_population}}\label{subsec:proof_prop_closed_form}
Recall that we defined the functions $f: \RR^d_+ \rightarrow \RR$ and $f_\sigma: \RR^{d \times n} \rightarrow \RR$ to be such that $f_P(w,x) = f_\sigma(X) = f(\sigma(X)).$ 
It is known that for constants $c_1, c_2 \in \RR_+$ we have that $c_1 b_1 + c_2 b_2 \stackrel{(d)}{=} \sqrt{c_1^2 + c_2^2} b_1,$ where $\stackrel{(d)}{=}$ denotes equality in distribution. Then 
\begin{align*}
f(s_1, s_2, 0, \dots, 0) & = \EE (|s_1 a_1 b_1 + s_2 a_2b_2|)  \\ 
& = \EE  \left(\EE (|s_1 a_1 b_1 + s_2 a_2b_2| \mid a_1, a_2) \right)\\ 
& = \EE  \left(\EE (\sqrt{(s_1 a_1)^2  + (s_2 a_2)^2}|b_1| \mid a_1, a_2) \right) \\ 
& = \sqrt{\frac{2}{\pi}}\EE\sqrt{(s_1 a_1)^2  + (s_2 a_2)^2} \\
& = \frac{\sqrt{\pi}}{\sqrt{2}\pi^2} \int_{- \infty}^{\infty} \int_{- \infty}^{\infty} {\sqrt{(a_1s_1)^2 + (a_2s_2)^2}} \exp\left(-\frac{a_1^2 + a_2^2}{2}\right)da_1 da_2 \\
& = 4\frac{\sqrt{\pi}}{\sqrt{2}\pi^2} \int_{0}^{\infty} \int_{0}^{\infty} {\sqrt{(a_1s_1)^2 + (a_2s_2)^2}} \exp\left(-\frac{a_1^2 + a_2^2}{2}\right)da_1 da_2 \\
& = 2\frac{\sqrt{2\pi}}{\pi^2} \int_{0}^{\infty} \int_{0}^{\pi/2} r^2{\sqrt{s_1^2\cos^2\theta + s_2^2\sin^2\theta}} \exp\left(-\frac{r^2}{2}\right)d\theta dr \\
& = \frac{2}{\pi} \int_{0}^{\pi/2}{\sqrt{s_1^2\cos^2\theta + s_2^2\sin^2\theta}} d\theta \\
& = \frac{2s_1}{\pi} \int_{0}^{\pi/2}{\sqrt{\cos^2\theta + \frac{s_2^2}{s_1^2}\sin^2\theta}} d\theta \\
& = \frac{2s_1}{\pi} \int_{0}^{\pi/2}{\sqrt{1 - \left(1 -\frac{s_2^2}{s_1^2}\right)\sin^2\theta}} d\theta \\
& = \frac{2s_1}{\pi} E\left(1 -\frac{s_2^2}{s_1^2}\right) \\
\end{align*}
where $E(\cdot)$ is the complete elliptic integral of the second kind (with parameter $m = k^2$). Thus altogether we obtain 
\[f_\sigma(X) = \sigma_{\max}(X) \sum_{n=0}^\infty \left(\frac{(2n)!}{2^{2n}(n!)^2}\right)^2\frac{
\left(1- \kappa^{-2}(X) \right)^n}{1-2n}\]
where $\kappa(X) = {\sigma_{\max}(X)}/{\sigma_{\min}(X)}$ is the condition number of $X.$

 \section{Proof of Theorem~\ref{cor:general-landscape}}\label{proof:generalPopulation}

 The proof of this result builds upon the next three lemmas. We will prove these lemmas and before we dive into the proof. Recall that $U \in O(d_1)$ and $V \in O(d_2)$ are any pair of matrices for which $X = U \sigma(X) V = \sum_i \sigma_i(X) U_i V_i^\top.$

 \begin{lemma}\label{lemma:stationary1}
The following are true.
\begin{enumerate}
 	\item \textbf{Anticorrelation.} The next equalities hold\[\dotp{U_1, w} \dotp{x,V_2} = \dotp{U_1, \bar w}\dotp{\bar x, V_2}\qquad \text{and} \qquad \dotp{U_2, w} \dotp{x,V_1} = \dotp{U_2, \bar w}\dotp{\bar x, V_1}. \]
 	\item \textbf{Singular values.} The singular values of $X$ satisfy 
 	\begin{equation*}
 	\arraycolsep=1.4pt\def\arraystretch{1.5}
 	\begin{array}{c}
 	\sigma_1(X) = \dotp{U_1, w}\dotp{x, V_1} - \dotp{U_1, \bar w} \dotp{\bar x, V_1} \geq 0,\\ \sigma_2(X) = \dotp{U_2, w}\dotp{x, V_2} - \dotp{U_2, \bar w} \dotp{\bar x, V_2} \geq 0.
 	\end{array}
 	\end{equation*}
 	\item \textbf{Correlation.} Assume that $\sigma_2(wx^\top - \bar w \bar x^\top) > 0,$ then $\spann\{x,\bar x\} = \spann\{V_1, V_2\}$, $\spann\{w,\bar w\} = \spann\{U_1, U_2\}$, and consequently,
 	\begin{align*}
 	\dotp{w,\bar w} &= \dotp{U_1, w} \dotp{U_1, \bar w} + \dotp{U_2, w}\dotp{U_2, \bar w},\\ 
 	\dotp{x,\bar x} &= \dotp{V_1, x} \dotp{V_1, \bar x} + \dotp{V_2, x} \dotp{V_2, \bar x}.
 	\end{align*} 
 \end{enumerate}  
\end{lemma}
\begin{proof}
The first equality in item one follows by observing that $U_1^\top X V_2 = 0$, expanding the expression on the left-hand-side gives the result. The same argument starting from $U_2^\top X V_1 = 0$ gives the other equality. The second item follows by definition. 

To prove the last item note that 
\[\dotp{U_i, w} x - \dotp{U_i,\bar w} \bar x = X^\top U_i = \sigma_i(X) V_i\qquad \forall i \in \{1,2\}.\]
Dividing through by $\sigma_i(X)$ in the previous inequality shows that $\spann\{x,\bar x\} = \spann\{V_1, V_2\}$. Therefore, we can write $x = \dotp{x, V_1} V_1 + \dotp{x, V_2} V_2$ and $\bar x = \dotp{\bar x, V_1} V_1 + \dotp{\bar x, V_2} V_2.$ Hence, 
\[\dotp{x,\bar x} = \left\langle \dotp{x, V_1} V_1 + \dotp{x, V_2} V_2, \dotp{\bar x, V_1} V_1 + \dotp{\bar x, V_2} V_2\right\rangle = \dotp{V_1, x} \dotp{V_1, \bar x} + \dotp{V_2, x} \dotp{V_2, \bar x}\]
An analogous argument shows the statement for $w$ and $\bar w$.   
\end{proof}
\begin{lemma}\label{lemma:stationary2}
The following hold true.
\begin{enumerate}
 	\item \textbf{Maximum correlation}.\begin{equation} \label{ineq:max_correlation}
 	\arraycolsep=1.4pt\def\arraystretch{1.5}
 	\begin{array}{c}
 	\max\{|\sigma_1 (Y) \dotp{v_1, x}|, |\sigma_2(Y) \dotp{v_2, x}|\} \leq \|Yx\|, \\ \max\{|\sigma_1 (Y) \dotp{u_1, w}|, |\sigma_2(Y) \dotp{u_2, w}|\} \leq \|Y^\top w\|.
 	\end{array}
 	\end{equation}
 	\item \textbf{Objective gap}. \begin{equation}\label{ineq:subgradient_singular_values} g(w,x) - g(\bar w, \bar x) \leq \sigma_1(Y)\sigma_1(X) + \sigma_2(Y) \sigma_2(X).\end{equation}
 \end{enumerate} 
\end{lemma}
\begin{proof}
Note that $\|Yx\| \geq \dotp{z, Yx}$ for all $z \in \SS^{d-1}$, then the very 
first claim follows by testing with $z \in \{\pm U_1, \pm U_2\}.$ An analogous 
argument gives the statement for $ w.$ 
Recall that $f$ is convex, consequently $f_\sigma$ is convex and the subgradient inequality gives
\[g(w,x) - g(\bar w, \bar x) = f_\sigma(X) - f_\sigma(0) \leq \dotp{Y, X} = \sigma_1(Y) \sigma_1(X) + \sigma_2(Y) \sigma_1(X).\]
\end{proof}
\begin{lemma}\label{lemma:rank1char}
Assume $\bar w \in \RR^{d_1 }$ and $\bar x \in \RR^{d_2}$ are nonzero vectors. Set $X = wx^\top + \bar w \bar x^\top$, then $X$ is a rank $1$ matrix if, and only if, $w = \lambda \bar w$ or $x = \lambda \bar x$ for some $\lambda \in \RR.$
\end{lemma}
\begin{proof}
It is trivial to see that if the later holds then $X$ is rank $1.$ Let us prove the other direction. Notice that if any of the vectors is zero we are done, so assume that none of them is. Recall that all the columns of $X$ are span from one vector. Consider the case where $x$ and $\bar x$ have different support (i.e. set of nonzero entries), then it is immediate that $w$ and $\bar w$ have to be multiples of each other.

Now assume that this is not the case, without loss of generality assume that $w \notin \spann\{\bar w\}$ and $x$ and $\bar x$ are nonzero and their first component is equal to one. Then the first column of $X$ is equal to $w + \bar w$, furthermore the second column is equal to $x_2w + \bar x_2 \bar w$ has to be a multiple of the first one. By assumption $w, \bar w$ are linearly independent therefore $x_2 = \bar x_2.$ Using the same procedure for the rest of the entries we obtain $x = \bar x.$
\end{proof}

We are now in good shape to describe the landscape of the function $g.$
\begin{proof}[Proof of Theorem~\ref{cor:general-landscape}]
To prove that at least one of the conditions hold we will show that if the first two don't hold then at least one of the other two have two hold. Assume that that the first two conditions are not satisfied, therefore $g(w,x) > g(\bar w, \bar x)$ and $(w,x) \neq (0,0).$ Let us furnished some facts before we prove this is the case. Notice that from~\eqref{ineq:subgradient_singular_values} we can derive
\[0 < \sigma_1(Y) \sigma_1(X) + \sigma_2(Y) \sigma_1(X) \leq 2 \sigma_1(Y{})\sigma_1(X),\]
thus $\sigma_1(Y), \sigma_1(X) > 0.$ On the other hand, since $(w,x)$ is critical inequalities~\eqref{ineq:max_correlation} immediately give 
\begin{align} \label{eq:sigma2is0} \begin{split}
&\sigma_1(Y)\dotp{V_1, x} = \sigma_2(Y) \dotp{V_2, x} = 0, \qquad \text{and} \qquad \sigma_1(Y)\dotp{U_1, w} = \sigma_2(Y) \dotp{U_2, w} = 0.   
\end{split}
\end{align}
So $\dotp{V_1, x} = 0$ and $\dotp{U_1, w} = 0$, then the first claim in Lemma~\ref{lemma:stationary1} gives. Additionally, this and the second claim in Lemma~\ref{lemma:stationary1} imply that 
\begin{align*}
\dotp{U_1, \bar w} \dotp{\bar x, V_2} = \dotp{U_2, \bar w}\dotp{\bar x, V_1} = 0,\qquad\text{and}\qquad-\dotp{U_1, \bar w} \dotp{\bar x, V_1} = \sigma_1(X) > 0.
\end{align*}
Combining these two gives $\dotp{U_2, \bar w} = \dotp{\bar x, V_2} = 0.$ Then by applying the second claim in Lemma~\ref{lemma:stationary1} we get $\sigma_2(X) = \dotp{U_2, w} \dotp{x , V_2}$. Using Equations~\eqref{eq:sigma2is0} we conclude that $\sigma_2(Y)\sigma_2(X) = 0.$ 

Now we will show that at least one of the conditions holds, depending on the value of $\sigma_2(X),$ let us consider two cases: 

\textbf{Case 1.} Assume $\sigma_2(X) = 0.$ This means that $X = wx^\top - \bar w \bar x^\top$ is a rank $1$ matrix. By Lemma~\ref{lemma:rank1char} we have that $w = \lambda \bar w$ or $x = \lambda \bar x$ for some $\lambda \in \RR.$ Note that if $w = \lambda \bar w$ then $U_1 = \pm \bar w/ \|\bar w\|$, then using Equation~\ref{eq:sigma2is0} we get that $\lambda \|\bar w\| = 0.$ Which implies that $\lambda = 0,$ and consequently $wx^\top = 0.$ An analogous argument applies when $x = \lambda \bar x.$ By assumption we have that $Yx = 0$ and $Y^\top w =0$. Additionally, since $X = -\bar w \bar x^\top$ we get that that $U_1 = \pm \bar w/ \|\bar w\|$ and $V_1 = \pm\bar x / \|\bar x\|.$ Recall that $Y = U \diag (\sigma (Y)) V^\top,$ then using the fact that $(w,x)$ is critical we conclude $\dotp{w,\bar w} = \dotp{x, \bar x} = 0.$ Implying that property three holds.

\textbf{Case 2.} Assume $\sigma_2(X) \neq 0.$ This immediately implies that $\sigma_2(Y) = 0.$ By the third part of Lemma~\ref{lemma:stationary1} we get that \[\dotp{x, \bar x} = \dotp{V_1, x} \dotp{V_1, \bar x} + \dotp{ V_2, x}\dotp{V_2, \bar x} = 0\]and analogously $\dotp{w, \bar w} = 0.$ Moreover, since $w \perp \bar w$ and $x \perp \bar w$ (and none of them are zero by assumption) we get that $(w/\|w\|, x/\|x\|)$ and $(\bar w/\|\bar w\|, \bar x/\|\bar x\|)$ are pairs of left and right singular vectors, with associated singular values $w^\top X x = \|w x^\top \|$ and $\bar w^\top X \bar x = \|\bar w \bar x^\top \|$, respectively. Assume that $\| w x^\top \| \geq \|\bar w \bar x^\top \|,$ thus 
$0 = w^\top Y x = \|wx^\top\| \sigma_1(Y) > 0,$ yielding a contradiction. Hence the condition four holds true.

Finally, we will prove the reverse statement. Assume that $(\bar w, \bar x)$ minimize $g.$ In this case, the set of points that satisfies the first conditions is the collection of minimizers so they are critical. Clearly $(w,x) = 0$ is always a stationary point, since $\|Y^\top w\| = \|Y x\| = 0$. Now let's construct a certificate $Y \in \partial f_\sigma(X)$ that ensures criticality for the remaining cases. 

Assume that $(w,x)$ that $wx^\top = 0,$ without loss of generality let's assume that $w = 0.$ Further, assume that there exists $Y \in \partial f_\sigma (w,x)$ such that $Yx = 0$ and $\dotp{x, \bar x} = 0.$ It is immediate that $(w,x)$ is a stationary point.

Assume that $(w,x)$ is such that $0 <\|wx^\top\| < \|\bar w \bar x^\top\|,$ $\dotp{w,\bar w} = \dotp{x,\bar x} = 0 $ and there exists $Y \in \partial f_\sigma(X)$ with $\sigma_2(Y) =0.$ By our argument above since $\|wx^\top\| < \|\bar w \bar x^\top\|$, any pair of admissible matrices $U,V$ satisfy $U_1 = \pm \bar w/ \|\bar w\|$ and $V_1 = \pm \bar x/ \|\bar x\|.$ Therefore 
		\[Y x = (\sigma_1(Y) U_1 V_1^\top) x = \pm \frac{\sigma_1(Y)}{\|\bar x\|} \dotp{\bar x, x} U_1 = 0, \]
		analogously $Y^\top w = 0.$
\end{proof}
\subsection{Proof of Lemma~\ref{lemma:partial_derivatives}}
\label{subsec:proof_lemma_partial_derivatives}
It is well-known that if $(a_1, a_2, \dots, a_d)$ is fixed (i.e. if we conditioned on it), then 
\[\sum_{i=1}^d a_i b_i s_i \stackrel{(d)}{=}\left(\sum_{i=1}^d (a_i s_i)^2\right)^{\frac{1}{2}} b\]
and $b$ is a standard normal random variable independent of the rest of the data. Therefore 
\begin{align*}
 f(s_1, \dots, s_d) & = \EE \left(\left|\sum_{i=1}^d a_i b_i s_i \right|\right)  = \EE \left(\EE\left( \left|\sum_{i=1}^d a_i b_i s_i \right| \Big| a_1, \dots, a_d\right) \right)\\ & = \EE \left(\left(\sum_{i=1}^d (a_i s_i)^2\right)^{\frac{1}{2}} \EE\left( b \mid a_1, \dots, a_d\right) \right) = \sqrt{\frac{2}{\pi}} \EE \left(\sum_{i=1}^d (a_i s_i)^2\right)^{\frac{1}{2}}.
 \end{align*}
Now, we need a technical tool in order to procede. 
 \begin{theorem}[\textbf{Leibniz Integral Rule}, Theorem 5.4.12 in \cite{MR1934675}]\label{theo:leibniz}
 Let $U$ be an open subset of $\RR^d$ and $\Omega$ be a measure space. Suppose that the function $h : U \times \Omega \rightarrow \RR$ satisfies the following: 
 \begin{enumerate}
 	\item For all $x \in U$, the function $h(x, \cdot)$ is Lebesgue integrable.
 	\item For almost all $w \in \Omega$, if we define $h^\omega(\cdot) = f(\cdot, \omega)$ the partial derivatives $\frac{\partial h^\omega}{\partial x_i} (x)$ exists for all $x \in U$.
 	\item There is an integrable function $\Phi : \Omega \rightarrow \RR$ such that $|\frac{\partial h^\omega}{\partial x_i} (x)| \leq \Phi(\omega)$ for all $x \in U$ and almost every $\omega \in \Omega.$
 \end{enumerate}
 Then, we have that for all $x \in U$ 
 \[\frac{\partial }{\partial  x_i} \int_\Omega h(x,\omega) d\omega = \int_\Omega \frac{\partial h^\omega}{\partial x_i}(x) d\omega. \]
 \end{theorem}
 
 This theorem tell us that we can swap partial derivatives and integrals provided that the function satisfies all the conditions above. Consider $\Omega$ to be the set $\RR^d$ endow with the Borel $\sigma$-algebra and the multivariate Gaussian measure. Define $h : \RR^d \times \Omega \rightarrow \RR$ to be given by $$(s , a) \mapsto \left(\sum_{i=1}^d (a_is_i)^2\right)^\frac{1}{2}.$$ Take $s \in \RR^d \setminus \{0\}$ to be an arbitary element, set $S = \{u \in \RR^d \mid \supp(s) \subseteq \supp(u) \}$, and define $U = B_\epsilon(s)$ with $\epsilon$ small enough such that $U \subseteq S$ and $\inf_{u \in U } \min_{i \in \supp(s)} |u_i| > 0$. Then it is easy to see that the first two conditions hold, in particular the second condition hold for all $a \neq 0$. Further, for any $x \in U$
 \begin{align*}\left|\frac{\partial h^a}{\partial s_j}(x)\right| = \left|\frac{a_j^2x_j}{\left(\sum_i^d (a_ix_i)^2\right)^{\frac{1}{2}}} \right| & \leq \frac{\sup_{u \in U} \|u\|_\infty}{\inf_{u \in U} \min_{i \in \supp(s)}  |u_i|} \frac{\sum_{i \in \supp(s)} a_j^2}{\left(\sum_{i \in \supp(s)} a_i^2\right)^{\frac{1}{2}}} \\ 
 & \leq \frac{\sup_{u \in U} \|u\|_\infty}{\inf_{u \in U} \min_i|u_i|}{\left(\sum_{i \in \supp(s)} a_i^2\right)^{\frac{1}{2}}},  \end{align*}
 where the last function is integrable with respect to the Gaussian measure. Thus, Theorem~\ref{theo:leibniz} ensures that the function $f$ is differentiable at every nonzero point. Consequently, for all $s \in \RR^d \setminus \{0\}$ 
 \begin{equation*}
  \frac{\partial f}{\partial s_j}(s) = \sqrt{\frac{2}{\pi}}s_j \; \EE \frac{a_j^2}{\left(\sum_i^d (a_is_i)^2\right)^{\frac{1}{2}}}.
 \end{equation*}
\section{Approximate critical points of a spectral function family}\label{app:approximate}

In Section~\ref{sec:population} we characterize the points for which $0 \in \partial f_P(w,x).$ In order to derive similar results for $f_S$ we will need to understand $\varepsilon$-critical points of $f_P,$ i.e. points $(w,x)$ for which $\dist(0, \partial f(w,x)) \leq \varepsilon.$ Just as before we adopt a more general viewpoint and consider spectral functions of the form $g(w,x) = f \circ \sigma (wx^\top - \bar w \bar x^\top).$

The main result in this section is Theorem~\ref{theo:epsilon_critical}. Given the fact that we don't have second order information in the form of a Hessian, we need to appeal to a different kind of growth condition. Turns out that the natural condition for this problem is
 \begin{equation} \label{eq:grow_cond}
 g(w,x) - g(\bar w, \bar x) \geq \kappa \norm{wx^\top - \bar w \bar x^\top}_F  \qquad \forall (w,x) \in \RR^{d_1}\times \RR^{d_2},
 \end{equation}
for some $\kappa > 0.$ Intuitively this means that the function grows sharply away from minimizers.

Before we dive into the main theorem, let us provide some technical lemmas. 
\begin{lemma}\label{lemma:lower_bound_singular_values}
Suppose there exists a constant $\kappa > 0$ such that \eqref{eq:grow_cond} holds. Then, for any point $(w,x)$ such that $wx^\top \neq \bar w \bar x^\top$ we have $\sigma_1(Y)+\sigma_2(Y) \geq \kappa. $
\end{lemma}
\begin{proof}
By definition $\sigma_2(X) \leq \sigma_1(X) \leq \norm{wx^\top - \bar w \bar x^\top}_F$. Then, applying~\eqref{ineq:subgradient_singular_values} gives 
\begin{align*}\kappa \norm{wx^\top - \bar w \bar x^\top}_F \leq g(w,x) - g(\bar w, \bar x) &\leq \sigma_1(Y) \sigma_1(X) + \sigma_2(Y)\sigma_2(X)\\ & \leq (\sigma_1(Y) + \sigma_2(Y))\norm{wx^\top - \bar w \bar x^\top}_F.\end{align*}
\end{proof}
\begin{lemma}
Suppose there exists a constant $\kappa > 0$ such that \eqref{eq:grow_cond} holds. Then any pair $(w,x) \in \RR^{d_1 +d_2} \setminus \{0\}$ satisfies 
\[\frac{1}{{\min\{\norm{w},\norm{x}\}}}\left(\kappa \norm{wx^\top - \bar w \bar x^\top} - (\sigma_1(Y) + \sigma_2(Y))\norm{\bar w \bar x^\top} \right) \leq \dist(0; \partial g(w,x)).\]
\end{lemma}
\begin{proof}
Notice that the result holds trivially for any pair such that $wx^\top = \bar w \bar x^\top.$ Let's assume that this is not the case. Recall that $\partial g(w,x) = \partial f_\sigma(X) x \times (\partial f_\sigma(X))^\top w.$ Pick $Y \in \partial g(w,x)$ such that $\dist(0, \partial g(w,x)) = \sqrt{\|Y x\|^2 + \|Y^\top w\|^2}.$ Using the convexity of $f_\sigma$ we get 
\begin{align*}
 \kappa \| wx^\top - \bar w \bar x\|_F \leq g(w,x) - g(\bar w, \bar x) & = f_\sigma(X) - f_\sigma(0) \\& \leq \dotp{Y,wx^\top - \bar w \bar x^\top}\\& \leq \|x\|\|Y^\top w\| + |w^\top Y x|
  \leq \|x\|\dist(0, \partial g(w,x)) + |w^\top Y x|,
 \end{align*}
 where the last inequality follows by Cauchy-Schwartz. Applying the same argument using $w^\top Yx \leq \|w\|\|Yx\|$ gives 
 \[ g(w,x) - g(\bar w, \bar x) \leq \min\{\|w\|,\|x\|\}\dist(0, \partial g(w,x)) + |w^\top Y x|.\] 
 Now, let's bound the second term on the right-hand-side. Note that
 \[|\bar w^\top Y \bar x| = |\dotp{Y,wx^\top}| \leq \|Y\|\|wx^\top\| \leq (\sigma_1(Y) + \sigma_2(Y))\norm{\bar w \bar x^\top}.\]
 The result follows immediately.
\end{proof}

We can now prove the main result of this section, a detailed location description of $\varepsilon$-critical points. This can be thought of as a quantitative version of Corollary~\ref{cor:general-landscape}. Its proof is however more involved due to the inexactness of the assumptions.
\begin{theorem}\label{theo:epsilon_critical}
Assume that $\|\bar w\| = \|\bar x\|$ and that there exists a constant $\kappa > 0$ such that \eqref{eq:grow_cond} holds. Further assume that $\sigma_1(Y)$ is bounded by some numerical constant.\footnote{This is implied for example when $f$ is Lipschitz.} Let $\zeta = (Yx, Y^\top w) \in \partial g(w,x)$, and set $\varepsilon = \|\zeta\|.$ Then if $wx^\top =0$ we have that \[\max\{\|Yx\|, \|Y^\top w\|\} \leq \varepsilon, \qquad \text{and} \qquad  \left\{\begin{array}{ll}
	|\dotp{w,\bar w}|& \lesssim \varepsilon\|\bar w\| \\
	|\dotp{x,\bar x}|& \lesssim \varepsilon\|\bar x\| 
	\end{array}  \right. . \] 
On the other hand, if $wx^\top \neq 0$ and $\|(w,x)\| \leq \nu \|(\bar w, \bar x)\|$ for some fixed $\nu > 1$. There exists a constant\footnote{Independent of $\nu$.} $\gamma > 0$ such that if $\varepsilon \leq \gamma\max\{\|w\|, \|x\|\}$ then $\norm{wx^\top} \lesssim \norm{\bar w \bar x^\top}$ and at least one of the following holds
\begin{enumerate}
	\item $$\max\{\norm{w}, \norm{x}\}\norm{wx^\top - \bar w \bar x^\top} \lesssim \varepsilon \norm{\bar w \bar x^\top}$$ 
	\item $$\min\{\|w\|, \|x\|\} \lesssim \varepsilon \qquad \text{and} \qquad \left\{\begin{array}{ll}
	|\dotp{w,\bar w}|& \lesssim \nu^2 \varepsilon\|\bar w\| \\
	|\dotp{x,\bar x}|& \lesssim \nu^2 \varepsilon\|\bar x\| 
	\end{array}  \right..
	$$
	\item $$\sigma_2(Y) \lesssim \frac{\varepsilon}{\max\{\|w\|, \|x\|\}}   \qquad \text{ and } \qquad  \left\{\begin{array}{ll}
	|\dotp{w,\bar w}|& \lesssim \nu^2 \varepsilon\|\bar w\| \\
	|\dotp{x,\bar x}|& \lesssim \nu^2 \varepsilon\|\bar x\| 
	\end{array}  \right. .$$
\end{enumerate}
\end{theorem}
\begin{proof}
First assume that $wx^\top = 0,$ then it is clear that $\max\{\|Yx\|, \|Y^\top w\|\} = \|\zeta\| \leq \varepsilon$. Without loss of generality assume that $x = 0,$ let $Y = U \sigma(Y) V^\top$ be the singular value decomposition. Since $X = -\bar w \bar x^\top$ then $U_1 = \pm \bar w/ \|\bar w\|$ and $V_1 = \pm \bar x / \norm{\bar x}$ and so
\begin{equation}\label{eq:small_innerproduct} \varepsilon \geq \|Y^\top w\| = \norm{\frac{\sigma_1(Y)}{\|\bar w\|} \dotp{{\bar w}, w}V_1 + z} \geq \frac{\sigma_1(Y)}{\|\bar w\|} \abs{\dotp{{\bar w}, w}} \geq \frac{\kappa}{2\norm{\bar w}} \dotp{{\bar w}, w}\end{equation}
where $z$ is orthogonal to $V_1$ and the second inequality follows by Lemma~\ref{lemma:lower_bound_singular_values}. This proves the first statement in the theorem. 

We know move to the ``On the other hand'' statement, assume $wx^\top \neq 0$ and $\|(w,x)\| \leq \nu \|(\bar w, \bar x)\|$. Notice that the result holds immediately if $(w,x) \in \{(\alpha \bar w, \bar x/\alpha) \mid \alpha \in \RR\}$. Further, due to Theorem~\ref{cor:general-landscape} it also holds when $\varepsilon = 0$. Let us assume that none of these two conditions are satisfied. 

We will start by showing that $\norm{wx^\top} \lesssim \norm{\bar w \bar x^\top}.$ Set \begin{equation}
	\delta = \frac{\sqrt{2}}{\kappa}(\sigma_1(Y)+\sigma_2(Y)) + 1.
\end{equation}
We showed in Lemma~\ref{lemma:lower_bound_singular_values} that $(\sigma_1(Y)+\sigma_2(Y)) \geq \kappa$ and thus $\delta > 1.$
\begin{claim}
The inequality $\norm{wx^\top} \leq \delta \norm{\bar w \bar x^\top}$ holds true.  
\end{claim}
\begin{proof}
Seeking contradiction assume that this is not the case. By the previous lemma 
\begin{equation}\label{eq:distrel}\frac{\sqrt{2}}{{\norm{wx^\top}}}\kappa \norm{wx^\top - \bar w \bar x^\top} - \frac{\varepsilon}{\max\{\|w\|, \|x\|\}} \leq  (\sigma_1(Y) + \sigma_2(Y))\frac{\norm{\bar w \bar x^\top}}{\norm{wx^\top}} .\end{equation}
Notice that 
\[\frac{\sqrt{2}}{{\norm{wx^\top}}}\kappa \norm{wx^\top - \bar w \bar x^\top} = {\sqrt{2}}\kappa \norm{\frac{wx^\top}{{\norm{wx^\top}}} - \frac{\bar w \bar x^\top}{{{\norm{wx^\top}}}}} \geq \sqrt{2}\kappa \abs{1-\frac{1}{\delta}}. \]
If we set $\gamma < \frac{\sqrt{2} \kappa}{2} \abs{1-\frac{1}{\delta}}$ then we ensure $\frac{\varepsilon}{\max\{\|w\|,\|x\|\}}< \frac{\sqrt{2} \kappa}{2} \abs{1-\frac{1}{\delta}},$ thus
\begin{align*}\frac{\sqrt{2} \kappa \abs{1-\frac{1}{\delta}}}{2(\sigma_1(Y)+\sigma_2(Y))} &\leq \frac{1}{(\sigma_1(Y)+\sigma_2(Y))}\left(\frac{\sqrt{2}}{{\norm{wx^\top}}}\kappa \norm{wx^\top - \bar w \bar x^\top} - \frac{\varepsilon}{\max\{\|w\|, \|x\|\}} \right) \\ & \leq \frac{\norm{\bar w \bar x^\top}}{\norm{w x^\top}} < \frac{1}{\delta}.\end{align*}
Rearranging we get
\[|\delta -1| < \frac{\sqrt{2}}{\kappa} (\sigma_1(Y)+\sigma_2(Y)),\]
leading to a contradiction.
\end{proof}

We now move on to proving that at least one of the three conditions has to hold. To this end, define 
\[\rho_1 = \frac{\max\{\|w\|,\|x\|\}}{\sqrt{2}} \qquad \text{and} \qquad  \rho_2 = \frac{1}{\kappa}\max\big\{2\sqrt{2}(1+\delta), 4\sigma_1(Y)\big\} \frac{\norm{\bar w\bar x^\top }}{\max\{\|w\|, \|x\|\}}.\]
Observe that if assume that if $\varepsilon \rho_2 \geq \norm{ wx^\top - \bar w \bar x^\top}$ then the result holds immediately. Assume from now on that $\varepsilon \rho_2 < \norm{ wx^\top - \bar w \bar x^\top}.$ Our road map is as follows, we will start by assuming $\min \{\|w\|, \|x\|\} \leq 2\varepsilon/\kappa$ and we will show that this implies the second condition in item two. Then we will move to assume that $\min \{\|w\|, \|x\|\} > 2\varepsilon/\kappa$ and show that item three has to hold.

Before we continue let us list some important facts. By Lemma~\ref{lemma:stationary2} 
\begin{equation} \label{ineq:critical_products}
	\max\{\sigma_1(Y)|\dotp{V_1,x}|, \sigma_2(Y)|\dotp{V_2,x}|, \sigma_1(Y)|\dotp{U_1,w}|, \sigma_2(Y) |\dotp{U_2,w}|\} \leq \varepsilon
\end{equation}
which together with $\sigma_1(Y) > \kappa/2$ implies that 
\begin{equation} \label{eq:innerproductsU1w}\max\{\abs{\dotp{U_1, w}},\abs{\dotp{V_1, x}}\} \leq \frac{\varepsilon}{\sigma_1(Y)} \leq \frac{2\varepsilon}{\kappa}.\end{equation}
Notice that this implies by Lemma~\ref{lemma:stationary1}
\begin{equation}\label{eq:crossedinnerproducts}|\dotp{U_1, \bar w}\dotp{\bar x, V_2}| = |\dotp{U_1, w} \dotp{x,V_2}| \leq \frac{2\|x\|\varepsilon}{\kappa}\qquad \text{and} \qquad \abs{\dotp{U_2, \bar w}\dotp{\bar x, V_1}} \leq \frac{2\|w\|\varepsilon}{\kappa}.\end{equation}
Observe that 
\begin{equation}\label{bound_individual_norms}\max\{\|w\|, \|x\|\} \leq \|(w,x)\| \leq \nu \|(\bar w, \bar x)\| = \sqrt{2} \nu \min\{\|\bar w\|, \|\bar x\|\}.\end{equation}

We can now continue with the proof. We will now assume that $\min \{\|w\|, \|x\|\} \leq 4\delta \varepsilon/\kappa$ and prove that item two holds.  
\begin{claim}
Assume that $\min \{\|w\|, \|x\|\} \leq 4\delta \varepsilon/\kappa.$ Then
\[|\dotp{w, \bar w}| \lesssim \nu^2 \varepsilon \|\bar w\| \qquad \text{ and } \qquad |\dotp{x, \bar x}| \lesssim \nu^2 \varepsilon \|\bar x\|.\]
\end{claim}
\begin{proof}
Notice 
\begin{align*}
\abs{\dotp{w, \frac{\bar w}{\|\bar w\|}}} \leq \abs{\dotp{w, \frac{\bar w}{\|\bar w\|} - U_1 }} + \abs{\dotp{w, U_1}} \leq \|w\| \norm{\frac{\bar w}{\|\bar w\|} - U_1} + \frac{2\varepsilon}{\kappa}
\end{align*}
where the last inequality follows by Cauchy-Schwartz and \eqref{eq:innerproductsU1w}. A similar argument gives the same bound with $\norm{\bar w/\|\bar w\| + U_1}$ instead. Now we need to make use of the Davis-Kahan Theorem.
\begin{theorem}
Let $A, \widehat A \in \RR^{d_1 \times d_2}$ with $\rank (A) = 1.$ Let $A = U\sigma(A)V^\top$ and $\widehat A = \widehat U \sigma( \widehat A) \widehat V^\top$ be their singular value decompositions. Then the 
\[\sin \theta(V_1, \widehat V_1) \leq \frac{2(2\sigma_1(A) + \|A- \widehat A\|_\op)}{\|A\|^2} \|A-\widehat A\|,\]
the same bound holds for $U_1, \widehat U_1.$
\end{theorem}
By letting $A = - \bar w \bar x^\top$ and $\widehat A = w x^\top- \bar w \bar x^\top $ in the previous theorem we get 
\begin{align*}
\min\left\{\norm{\frac{\bar w}{\|\bar w\|} + U_1}, \norm{\frac{\bar w}{\|\bar w\|} - U_1}\right\} & \leq \sqrt{2} \sin \left(\theta(\bar w/\|\bar w\|, U_1 )\right) \\ 
& \leq 2\sqrt{2} \frac{(2\|\bar w \bar x^\top\| + \|wx^\top\|)}{\|\bar w \bar x^\top\|^2} \|wx^\top\| \\
& \leq 2\sqrt{2}(2 + \delta) \frac{\|wx^\top\|}{\|\bar w \bar x^\top\|} \\
& \leq 2\sqrt{2}(2 + \delta)\nu \frac{\varepsilon}{\|\bar w \|} 
\end{align*}
where the last inequality follows since $\|wx^\top\|\leq \varepsilon \max\{\|w\|, \|x\|\}$ and \eqref{bound_individual_norms}. Hence from the previous inequalities we derive 
\[\abs{\dotp{w, \frac{\bar w}{\|\bar w\|}}} \leq \|w\| \norm{\frac{\bar w}{\|\bar w\|} - U_1} + \frac{2\varepsilon}{\kappa}  \leq 2\sqrt{2}(2 + \delta) \nu\frac{\|w\|}{\|\bar w \|} \varepsilon + \frac{2 \varepsilon}{\kappa} = \left(2\sqrt{2}(2 + \delta) \nu^2 + \frac{2 }{\kappa}\right)\varepsilon.\]
A completely analogous result holds for $|\dotp{x, \bar x}|.$
\end{proof}

Suppose now that $\min\{\|w\|, \|x\|\}> 4\delta \varepsilon/\kappa.$ In the remainder of the proof we will show that in this case, item three has to hold.
\begin{claim}	
The rank of $X = wx^\top - \bar w \bar x^\top$ is two.
\end{claim}	
\begin{proof}
Assume $w = \lambda \bar w$, then $U_1 = \pm w/\|\bar w\|$ and using the same computation as in Equation~\eqref{eq:small_innerproduct} we get $\lambda \|\bar w\| \leq 2\epsilon/\kappa \leq 4\delta \epsilon/\kappa$ which implies $\min\{\|w\|, \|x\|\} \leq 4\delta \varepsilon/\kappa,$ yielding a contradiction. An analogous argument holds for $x = \lambda \bar x.$ Thus, Lemma~\ref{lemma:rank1char} implies that $\sigma_2(wx^\top - \bar w \bar x^\top) > 0.$ 
\end{proof}

\begin{claim} \label{claim:bounds2Y}
$\sigma_2(Y) < \frac{\varepsilon}{\rho_1}.$
\end{claim}
\begin{proof}
Without loss of generality suppose $\|w\| = \max\{\|w\|,\|x\|\}.$ Assume seeking contradiction that this isn't true, thus $\sigma_2(Y) \geq \varepsilon/\rho_1$ then Inequality~\eqref{ineq:critical_products} gives $|\dotp{U_2, w}| \leq \rho_1.$ Furthermore, notice that due to Lemma~\ref{lemma:stationary1} we have that $\|w\|^2 = \dotp{U_1,w}^2 + \dotp{U_2, w}^2$ and consequently $|\dotp{U_1,w}| \geq \sqrt{\|w\|^2 - \rho_1^2}.$ Again, due to~\eqref{ineq:critical_products}
\[\sigma_1(Y) \leq \frac{\varepsilon}{|\dotp{U_1, w}|} \leq \frac{\varepsilon}{\sqrt{\|w\|^2- \rho_1^2}}.\] 
In turn this implies
\begin{align*}
\kappa \varepsilon \rho_2 < \kappa \norm{wx^\top - \bar w \bar x^\top} \leq g(w,x) - g(\bar w, \bar x) & \leq \sigma_1(Y) \sigma_1(X) + \sigma_2(Y)\sigma_2(X) \\ 
&\leq 2 \sigma_1(Y)\sigma_1(X) \\ 
& \leq 2\frac{\varepsilon}{\sqrt{\|w\|^2- \rho_1^2}} |\dotp{U_1, w}\dotp{x, V_1} - \dotp{U_1, \bar w} \dotp{\bar x, V_1}| \\ 
& \leq 2\frac{\varepsilon}{\sqrt{\|w\|^2- \rho_1^2}} \left(\norm{wx^\top} + \norm{\bar w \bar x^\top}\right) \\
& \leq \frac{2\sqrt{2}\varepsilon}{\|w\|}\left(1+\delta\right)\norm{\bar w \bar x^\top} 
\end{align*}
Rearranging we get
\[\rho <\frac{2\sqrt{2}(1+\delta)}{\kappa} \frac{\norm{\bar w\bar x^\top }}{\max\{\|w\|, \|x\|\}}, \]
yielding a contradiction. 
\end{proof}
We now need to prove an additional claim. 
\begin{claim}
$\abs{\dotp{U_2, \bar w}} \leq \abs{\dotp{U_1, \bar w}}$ and $\abs{\dotp{V_2, \bar x}} \leq \abs{\dotp{V_1, \bar x}}$
\end{claim}
\begin{proof}
Seeking contradiction we assume the possible contrary cases. 

\emph{Case 1.} Assume $\abs{\dotp{U_2, \bar w}} > \abs{\dotp{U_1, \bar w}}$ and $\abs{\dotp{V_2, \bar x}} > \abs{\dotp{U_1, \bar x}}$, then\eqref{eq:innerproductsU1w} and \eqref{eq:crossedinnerproducts} imply
\[\max\{|\dotp{U_1, w} \dotp{V_1, x}|,|\dotp{U_1, \bar w} \dotp{V_1, \bar x}|\} \leq \frac{2\min\{\|w\|, \|x\|\}\varepsilon}{\kappa}.\]
	From which we derive
	\begin{align*}
	\kappa \varepsilon \rho_2 < g(w,x) - g(\bar w , \bar x) \leq 2\sigma_1(Y)\sigma_1(X) \leq 4\sigma_1(Y) \delta \frac{\norm{\bar w \bar x^\top}}{\max\{\|w\|,\|x\|\}} \varepsilon.
	\end{align*}
	contradicting the definition of $\rho_2.$

\emph{Case 2.} Assume that $\abs{\dotp{U_2, \bar w}} \leq \abs{\dotp{U_1, \bar w}}$ and $\abs{\dotp{V_2, \bar x}} > \abs{\dotp{V_1, \bar x}}.$ Notice that $\|\bar w\|^2 = \dotp{U_1, \bar w}^2 + \dotp{U_2, \bar w}^2$, hence $\abs{\dotp{U_1, \bar w}} \geq \|\bar w\|/\sqrt{2}$ and similarly $\abs{\dotp{V_2, \bar x}} > \|\bar x\|/\sqrt{2}.$ Thus, 
	\[\frac{\|\bar w\|}{\sqrt{2}} \leq |\dotp{U_1, \bar w}|\leq  \frac{2\|x\|\varepsilon}{\kappa|\dotp{\bar x, V_2}| }<\frac{2\sqrt{2}\|x\|\varepsilon}{\kappa\|\bar x\| }.\]
	This implies that 
	\[ \min\{\|w\|,\|x\|\} \leq \|w\| \leq \delta \frac{\norm{\bar w \bar x^\top}}{\|x\|} < \frac{4\delta \varepsilon}{\kappa},\]
	yielding a contradiction.
\end{proof}

Without loss of generality let us assume $\|w\| \leq \|x\|.$
\begin{claim}
\[\abs{\dotp{w, \bar w}} \lesssim \varepsilon \norm{\bar w}  \qquad \text{and} \qquad \abs{\dotp{x, \bar x}} \lesssim  \nu^2 \varepsilon \norm{\bar x}. \]
\end{claim}
\begin{proof}
By the previous claim and the fact that $\|\bar w\|^2=\dotp{U_1, \bar w}^2 + \dotp{U_2, \bar w}^2$ we get that $|\dotp{U_1, \bar w}| \geq \|w\|/\sqrt{2}$, combining this with \eqref{eq:crossedinnerproducts} gives
\[|\dotp{\bar x, V_2}| \leq \frac{2\sqrt{2}\|x\|\varepsilon}{\kappa \|\bar w\|} \leq \frac{4\delta}{\kappa} \nu \varepsilon\]
Then by Lemma~\ref{lemma:stationary1} 
\begin{align*}
\abs{\dotp{x,\bar x}} = |\dotp{V_1, x} \dotp{V_1, \bar x} + \dotp{V_2, x}\dotp{V_2, \bar x}| & \leq |\dotp{V_1, x} \dotp{V_1, \bar x}| + |\dotp{V_2, x}\dotp{V_2, \bar x}| \\
& \leq \frac{2\varepsilon}{\kappa}\|\bar x\| + \|x\||\dotp{V_2, \bar x}|\\ & \leq \left(\frac{2}{\kappa} + \frac{4\delta}{\kappa}\nu^2\right)\varepsilon\|\bar x\| \leq \left(\frac{2}{\kappa} + \frac{4\delta}{\kappa}\nu^2\right)\varepsilon\|\bar x\| .
\end{align*}
where we used \eqref{bound_individual_norms}. Notice that the same analysis gives 
\[\abs{\dotp{w,\bar w}} \leq \left(\frac{2}{\kappa} + \frac{2\sqrt{2}\delta}{\kappa}\frac{\|w\|}{\|x\|}\right)\varepsilon\|\bar w\| \leq \left(\frac{2}{\kappa} + \frac{2\sqrt{2}\delta}{\kappa}\right)\varepsilon\|\bar w\|.\]
\end{proof}
This last claim finishes the proof of the theorem.
\end{proof}

\section{Proofs of Theorem~\ref{theo:sample_objective}}\label{app:Sample}
In order to prove the theorem we will apply three steps: we will show that the graphs of $\partial f_S$ and $\partial f_P$ are close, then use Theorem~\ref{theo:epsilon_critical} to study the $\epsilon$-critical points of $f_P$ and finally conclude about the landscape of $f_S$ by combining the previous two steps. The following two propositions handle the first part.

\begin{proposition}\label{prop:close_graphs}
Fix two functions $f, g: \RR^{d_1}\times \RR^{d_2} \rightarrow \RR$ such that $g$ is $\rho$-weakly convex. Suppose that there exists a point $(\bar w, \bar x)$ and a real $\delta > 0$ such that the inequality 
\[|f(w,x) - g(w,x)| \leq \delta \norm{wx^\top - \bar w \bar x^\top}_F \qquad \text{holds for all }(w,x) \in \RR^{d_1} \times \RR^{d_2}.\]
Then for any stationary point $(w,x)$ of $g$, there exists a point $(\widehat w, \widehat x)$ satisfying 
\[\left\{\begin{array}{ll}
\norm{(w,x) - (\widehat w, \widehat x)} &\leq 2\sqrt{\frac{\delta \norm{wx^\top - \bar w \bar x^\top}}{\rho+\delta}} \\
\norm{\dist(0, \partial f(\widehat w, \widehat x))} & \leq \left(\delta + \sqrt{2\delta(\rho + \delta)} \right)\left(\norm{(w,x)} + \norm{(\bar w ,\bar x)}\right).\end{array} \right.\]
\end{proposition}
\begin{proof}
The proposition is a corollary of Theorem 6.1 of \cite{davis2017nonsmooth}. Recall that for a function $l: \RR^d \rightarrow \overline \RR$ the Lipschitz constant at $\bar y \in \RR^d$ is given by $$ \lip(l,y) := \limsup_{y \rightarrow \bar y} \frac{|l(y)- l(\bar y)|}{|y - \bar y|}.$$ Set $u(x) = \delta \norm{wx^\top - \bar w \bar x^\top}$, and $l(x) = -\delta \norm{wx^\top - \bar w \bar x^\top}$. It is easy to see that at differentiable points the gradient of $l(\cdot)$ is equal to 
\[\nabla l(w,x) = -\frac{\delta}{\norm{wx^\top - \bar w \bar x^\top}_F }\begin{bmatrix}
(wx^\top-\bar w \bar x^\top) x \\
(wx^\top-\bar w \bar x^\top)^\top w 
\end{bmatrix} \implies \norm{\nabla l(w,x)} \leq \delta\|(w,x)\|\]
Then, since $\lip(l;w,x) = \limsup_{(w',x') \rightarrow (w,x)} \|\nabla l(w',x')\|$, we can over estimate $$\lip(l;w,x) \leq \delta\left( \norm{(w,x)} + \norm{(\bar w , \bar x)}\right).$$ Thus applying Theorem 6.1 of \cite{davis2017nonsmooth} we get that for all $\gamma > 0$ there exists $(\widehat w, \widehat x)$ such that $\|(w,x) - (\widehat w, \widehat x)\| \leq 2 \gamma$ and
\[\dist(0, \partial f(\widehat w, \widehat x)) \leq 2\rho \gamma + 2 \delta \frac{\norm{wx^\top - \bar w \bar x^\top}}{\gamma} + \delta\left(\norm{(\widehat w, \widehat x)} + \norm{(\bar w,\bar x)}\right) \]
By the triangular inequality we get $\norm{(\widehat w, \widehat x)} \leq 2\gamma + \norm{(w,x)}$ and therefore 
\[\dist(0, \partial f(\widehat w, \widehat x)) \leq 2(\rho + \delta) \gamma + 2 \delta \frac{\norm{wx^\top - \bar w \bar x^\top}}{\gamma} + \delta \norm{(w,x)}.\]
Hence setting $\gamma = \sqrt{\frac{\delta \norm{wx^\top - \bar w \bar x^\top}}{\rho+\delta}},$ gives
\begin{align*}
\dist(0, \partial f(\widehat w, \widehat x)) & \leq 2\sqrt{\delta(\rho + \delta)\norm{wx^\top - \bar w \bar x^\top}} + 
\delta \left(\norm{(\widehat w, \widehat x)} + \norm{(\bar w,\bar x)}\right)\\
& \leq 2\sqrt{\delta(\rho + \delta)\left(\norm{wx^\top} + \norm{\bar w \bar x^\top}\right)} + 
\delta \left(\norm{(\widehat w, \widehat x)} + \norm{(\bar w,\bar x)}\right)\\
& \leq 2\sqrt{\delta(\rho + \delta)}\left(\sqrt{\norm{wx^\top}} + \sqrt{\norm{\bar w \bar x^\top}}\right) + 
\delta \left(\norm{(\widehat w, \widehat x)} + \norm{(\bar w,\bar x)}\right)\\
& \leq \left(\delta + \sqrt{2\delta(\rho + \delta)} \right)\left(\norm{(w,x)} + \norm{(\bar w ,\bar x)}\right)
\end{align*}
where we used that $\sqrt{a+b} \leq \sqrt{a} + \sqrt{b}$ and $ab \leq (a^2+b^2)/2.$
\end{proof}

\begin{proposition}\label{prop:concentration}
There exist numerical constants $c_1, c_2 > 0$ such that for all $(w,x) \in \RR^{d_1 \times d_2}$ we have 
\begin{equation}\label{concentration}\Big| f_S(w,x) - f_P(w,x)\Big| \lesssim \left( \frac{d_1+d_2 +1}{m} \log\left(\frac{m}{d_1+d_2+1}\right) \right)^{\frac{1}{2}} \|wx^\top - \bar w \bar x^\top \| \;\;\end{equation} 
with probability at least $1- 2\exp(-c_1(d_1+d_2+1))$ provided $m \geq c_2(d_1+ d_2 +1).$
\end{proposition}
\begin{proof}
The proof of this proposition is almost entirely analogous to the proof of Theorem 4.6 in \cite{charisopoulos2019composite} (noting that Gaussian matrices satisfy the hypothesis of this result). The proof follows exactly the same up to Equation (4.19) in the aforementioned paper. Where the authors proved that there exists constants $c_1, c_2, c_3, c_4> 0$ such that for any $t \in (0,c_4)$ the following uniform concentration bound holds 
\begin{equation*} \Big| f_S(w,x) - f_P(w,x)\Big| \leq \frac{3}{2}t\|wx^\top - \bar w \bar x^\top \|_F \;\; \text{for all } (w,x) \in \RR^{d_1 \times d_2}\end{equation*}
with probability at least $1- 2\exp(c_1 (d_1+d_2+1 +1) \log (c_2 /t) - c_3 t^2 m).$ This probability is at least $1-2\exp(-c_3 t^2 m/2)$ provided that 
\begin{equation}
\label{locura}  \frac{d_1+d_2 +1}{m} \leq \frac{c_3t^2}{2c_1\log(c_2/t)}.
\end{equation}
Set $t = \max\left(\sqrt{\frac{2c_1}{c_3}},c_2\right)\left( \frac{d_1+d_2 +1}{m} \log\left(\frac{m}{d_1+d_2+1}\right) \right)^{\frac{1}{2}}$. This choice ensures that \eqref{locura} holds, since  
\begin{align*}
\frac{d_1+d_2+1}{m} \leq \frac{ (d_1+d_2 + 1) \log(m/(d_1+d_2+1))}{m \log \left( \frac{m }{d_1+d_2 +1 } \log^{-1}\left(\frac{m}{d_1+d_2+1}\right) \right)} \leq \frac{c_3 t^2}{2c_1\log(c_2/t)}.   
\end{align*}
We can ensure that $t \in (0,c_4)$ by setting $m \geq C(d_1+d_2+1)$ with $C$ sufficiently large. This proves the result (after relabeling the constants).  
\end{proof}

We are finally in position to proof the theorem. 

\begin{proof}[Proof of Theorem~\ref{theo:sample_objective}]
Fix $v \geq 1$ and a fix point $(w,x)$ satisfying $\|(w,x)\| \leq \nu \|(\bar w, \bar x)\|$. Proposition~\ref{prop:concentration} shows that there exist constants $c_1, c_2 > 0$ such that with probability at least $1- 2\exp(-c_1(d_1+d_2+1))$ we have 
\[ |f_S(w,x) - f_P(w,x)| \leq \tilde \cO\left(\left(\frac{d_1+d_2+1}{m} \right)^{\frac{1}{2}}\right) \|wx^\top - \bar w \bar x^\top\|_F \qquad \forall (w,x) \in \RR^{d_1} \times \RR^{d_2}\]
provided that $m \geq c_2(d_1+d_2+1).$ To ease the notation let us denote $\bar \Delta := \tilde \cO\left(\left(\frac{d_1+d_2+1}{m} \right)^{\frac{1}{2}} \right).$ Assume that we are in the event in which this holds, it is known that $f_S$ is $\rho$-weakly with high probability provided that $m \geq C (d_1+d_2 +1)$, see Section 3 and Theorem 4.6 in \cite{charisopoulos2019composite}. Now, assume that $m$ is big enough and we are in the intersection of this two events. This holds with probability $1-c_3\exp(c_1(d_1+d_2+1))$ (for some possibly different constants $c_1,c_3$). Hence by Proposition~\ref{prop:close_graphs} there exits a point $(\widehat w, \widehat x)$ such that 
\[\|(w,x) - (\widehat w, \widehat x) \| \leq \frac{2}{\sqrt{\rho}}\sqrt{\Delta} D_{wx}\qquad \text{and} \qquad \dist(0, \partial f(\bar w, \bar x)) \leq C\sqrt{\Delta} D_{wx}\]
where $D_{wx} = \|(w,x)\| + \|(\bar w, \bar x)\|.$ 

Notice that if $\|(w,x)\| \leq \bar \Delta^\frac{1}{4} \|(\bar w, \bar x)\|$ holds then the result holds immediately. So assume that this inequality is not satisfied. So we can lower bound 
\[\|(\widehat w, \widehat x)\| \geq \|(w,x)\| - \|(\widehat w, \widehat x) - (w,x)\| \geq \left( 1 - 2{\left(\frac{\bar \Delta}{\rho}\right)}^{\frac{1}{2}}\left(1+\bar \Delta^{-\frac{1}{4}}\right) \right)\|(w,x)\| \geq \frac{1}{2}\|(w,x)\|\]
where the first inequality follow by applying the triangle inequality and the last inequality follows for $m$ sufficiently large, since we can ensure that for such $m$ the term in the parenthesis is bigger than $1/2$. Therefore, 
\begin{align*}
\dist(0, \partial f(\widehat w,\widehat x)) & \leq C {\Delta}^\frac{1}{2}\left(\|(w,x)\| + \|(\bar w, \bar x)\|\right)\\
& \leq C {\Delta}^\frac{1}{2}\left(1 + \Delta^{-\frac{1}{4}}\right)\|(w,x)\| \\
& \leq 2C {\Delta}^\frac{1}{2}\left(1 + \Delta^{-\frac{1}{4}}\right)\|(\widehat w,\widehat x)\| \\ 
& \leq 4 C{\Delta}^\frac{1}{4}\|(\widehat w, \widehat x)\|. 
\end{align*}
Hence, by reducing $\bar \Delta$ if necessary we can guarantee that $\dist (0,\partial f(\widehat w,\widehat x)) \leq \gamma \|(\widehat w, \widehat x)\|$ and consequently Theorem~\ref{theo:epsilon_critical} gives that at least one of the following two holds 
\begin{equation}\label{last}\max\{\|\widehat w\|, \|\widehat x\| \} \|\widehat w \widehat x^\top - \bar w \bar x^\top \| \lesssim  {\Delta}^\frac{1}{2}D_{wx} \|\bar w\bar x^\top\| \qquad \text{and} \qquad \left\{\begin{array}{ll}
	|\dotp{w,\bar w}|& \lesssim \nu^2 {\Delta}^\frac{1}{2}D_{wx} \|\bar w\| \\
	|\dotp{x,\bar x}|& \lesssim \nu^2 {\Delta}^\frac{1}{2}D_{wx}\|\bar x\| 
	\end{array}  \right.\end{equation}
Let us prove that this implies the statement of the theorem. 

\textbf{Case 1.} Assume that the second condition in \eqref{last} holds. Then 
\begin{align*}
\abs{\dotp{w,\bar w}} \leq \abs{\dotp{\widehat w, \bar w}} + \|\bar w\| \|\widehat w - w\| \lesssim (\nu^2 + 1) \Delta^{\frac{1}{2}} D_{wx} \|\bar w\| \lesssim (\nu^2 + 1) \Delta^{\frac{1}{4}} \|(w,x)\| \|\bar w\|
\end{align*}
where we used that $\Delta^{\frac{1}{2}} D_{wx} \lesssim \Delta^{\frac{1}{4}} \|(w,x)\|$ for $m$ big enough. A similar argument yields the result for $|\dotp{w,x}|.$

\textbf{Case 2.} On the other hand, if the first condition holds, there exist $e_w \in \RR^{d_1}, e_x \in \RR^{d_2}$ such that $\widehat w = w + e_w$ and $\widehat x = x + e_x$ with $\|e_w\|, \|e_x\| \leq \Delta^{\frac{1}{2}} D_{wx}.$ Then
\begin{align*}
\|(w,x)\| \| wx^\top - \bar w \bar x^\top\| &\leq \|(w,x)\| \| wx^\top - \widehat w \widehat x^\top\|  + \|(w,x)\|\|\widehat w \widehat x^\top - \bar w \bar x^\top\| \\
& \leq \|(w,x)\| \| wx^\top - \widehat w \widehat x^\top\|  + 2\|(\widehat w,\widehat x)\|\|\widehat w \widehat x^\top - \bar w \bar x^\top\| \\
& \lesssim \|(w,x)\| \| wx^\top -  (w + e_w) ( x + e_x)^\top\|  + {\Delta}^\frac{1}{2}D_{wx} \|\bar w\bar x^\top\| \\
& \leq \|(w,x)\| \left(\|we_x^\top \| + \|e_w x^\top\|+  \|e_w e_x^\top\|\right)  + {\Delta}^\frac{1}{2}D_{wx} \|\bar w\bar x^\top\| \\
& \leq \|(w,x)\| {\Delta}^\frac{1}{2}D_{wx} \left(\|w\| + \| x^\top\|+  {\Delta}^\frac{1}{2}D_{wx} \right)  + {\Delta}^\frac{1}{2}D_{wx} \|\bar w\bar x^\top\| \\
& \lesssim \|(w,x)\| {\Delta}^\frac{1}{2}D_{wx} \left(\|(w,x)\|+  {\Delta}^\frac{1}{4}\|(w,x)\| \right)  + {\Delta}^\frac{1}{2}D_{wx} \|\bar w\bar x^\top\| \\
& \lesssim \|(w,x)\|^2 {\Delta}^\frac{1}{2}D_{wx}  + {\Delta}^\frac{1}{2}D_{wx} \|\bar w\bar x^\top\| \\
& \lesssim (\nu^2 + 1)\|\bar w \bar x^\top\| {\Delta}^\frac{1}{2}D_{wx}  \\
& \lesssim (\nu^2+1)\Delta^\frac{1}{4}\|( w,  x)\| \|\bar w \bar x^\top\|
.
\end{align*}
Proving the desired result.
\end{proof}
\end{document}